\documentclass[a4paper,11pt]{amsart}

\usepackage{pgf,tikz}
\usetikzlibrary{arrows}
\usepackage{amssymb,amscd,amsthm,mathrsfs}
\usepackage[ansinew]{inputenc}
\usepackage[centertags]{amsmath}
\usepackage{graphicx,psfrag}
\usepackage{xcolor}
\usepackage{cite}

\newtheorem{thm}{Theorem}[section]
\newtheorem{cor}[thm]{Corollary}
\newtheorem{lemma}[thm]{Lemma}
\newtheorem{prop}[thm]{Proposition}
\theoremstyle{definition}

\theoremstyle{remark}
\newtheorem{rem}[thm]{Remark}

\newcommand{\w}{\omega}

\newcommand{\abs}[1]{\left|#1\right|}

\newcommand{\T}{\mathbb{T}}

\newcommand{\N}{\mathbb{N}}

\newcommand{\I}{\mathbb{I}}

\newcommand{\Int}{\textrm{int}}

\newcommand{\Bcal}{\mathcal{B}}
\newcommand{\Ccal}{\mathcal{C}}
\newcommand{\Dcal}{\mathcal{D}}
\newcommand{\Hcal}{\mathcal{H}}
\newcommand{\Ocal}{\mathcal{O}}
\newcommand{\Scal}{\mathcal{S}}

\newcommand{\Vcal}{\mathcal{V}}

\newcommand{\mc}{\mathcal}
\newcommand{\eps}{\varepsilon}
\newcommand{\ssq}{\subseteq}
\renewcommand{\:}{\colon}
\DeclareMathOperator{\Cri}{Cri}
\DeclareMathOperator{\Per}{Per}
\DeclareMathOperator{\Tra}{Tra}
\DeclareMathOperator{\Aut}{Aut}

\title{The bifurcation set as a topological invariant for one-dimensional dynamics}

\author[Gabriel Fuhrmann]{Gabriel Fuhrmann}
\author[Maik Gr\"{o}ger]{Maik Gr\"{o}ger}
\author[Alejandro Passeggi]{Alejandro Passeggi}

\address[Gabriel Fuhrmann]{Department of Mathematics, Imperial College London, 180 Queen's Gate,
London SW7 2AZ, United Kingdom }
\email{gabriel.fuhrmann@imperial.ac.uk}

\address[Maik Gr\"{o}ger]{Faculty of Mathematics, University of Vienna, Oskar Morgensternplatz 1,
1090 Vienna, Austria}
\email{maik.groeger@univie.ac.at}

\address[Alejandro Passeggi]{Facultad de Ciencias, Universidad de la Rep\'{u}blica,
Igu\'{a} 4225 esq. Mataojo, Montevideo, Uruguay.}
\email{alepasseggi@gmail.com}

\thanks{This project has received funding from the European Union's Horizon 2020
research and innovation program under the Marie Sklodowska-Curie grant agreement
No 750865.
Further, MG acknowledges support by the DFG grants JA 1721/2-1 and GR 4899/1-1.
Moreover, this project has received funding from the CSIC project
SD 618.
Finally, the authors would like to thank Henk Bruin, Pablo Guarino, Sebastian van
Strien and Bj\"orn Winckler for  related discussions and many valuable comments.}

\begin{document}

\maketitle

\begin{abstract}
	For a continuous map on the unit interval or circle, we define the bifurcation
	set to be the collection of those interval holes whose surviving set is
	sensitive to arbitrarily small changes of their position.
	By assuming a global perspective and focusing on the geometric and topological
	properties of this collection rather than the surviving sets of individual holes,
	we obtain a novel topological invariant for one-dimensional dynamics.

	We provide a detailed description of this invariant in the realm of transitive maps
	and observe that it carries fundamental dynamical information.
	In  particular, for transitive non-minimal piecewise monotone maps,
	the bifurcation set encodes the topological entropy
	and strongly depends on the behavior of the critical points.
\end{abstract}

\section{Introduction}

Given some dynamical system on a topological space and an open subset (called
\emph{hole} in the following), it is natural to study the associated
\emph{surviving set}, that is, the collection of all points which never enter
this subset.
In this framework, the theory of open dynamical systems is, for instance, concerned
with escape rates, conditionally invariant measures and other closely related concepts,
see for example \cite{DemersYoung2005, KellerLiverani2009,Bunimovich2011, Dettmann2011,
AltmannPortelaTel2013,DemersFernandez2016,PollicottUrbanski2017,BruinDemersTodd2018}
for more information and further references.
Recently, there has been an increased interest in understanding families of suitably
parametrized interval holes of one-dimensional maps whose surviving sets fulfill
certain properties, see for instance \cite{CarminatiTiozzo2011,Sidorov2014,
HareSidorov2014,GlendinningSidorov2015,Groeger2015,Clark2016,KalleKongLangeveldLi2018}.
As a matter of fact, this thread of research goes back to the classical work by
Urba\'nski \cite{Urbanski1986}.

In this spirit,  we propose to study the family of all interval holes representing
distinct surviving dynamics as a source of topological invariants.
To be more precise, for a continuous map $f$ on the interval $[0,1]$ or the circle $\T$,
we consider the \emph{bifurcation set} $\Bcal_f$ which is given by all
those intervals whose surviving set can change under arbitrarily small perturbations.
To get a first impression of the bifurcation set, see Figure \ref{f.approx bif set doubling map}
below, where an approximation of $\Bcal_f$ for the doubling map on the circle is
depicted.

Before we state our main results, let us introduce some basic definitions.
Throughout this work, $\I$ refers to $[0,1]$ (in which case we set $\partial \I=\{0,1\}$)
or $\T$ (in which case $\partial \I=\emptyset$).
If $\I=[0,1]$, a hole is given by an open interval $(a,b)$ with
$a, b\in \I\setminus \partial \I$.\footnote{The assumption that $a$ and $b$ avoid the
boundary points $\{0,1\}$ simply reduces certain technicalities and is not of any further importance.
For an explicit study of general continuous maps on $[0,1]$ with interval holes
of the form $[0,t)$ and $(t,1]$ where $t\in[0,1]$, see \cite{Groeger2015}.}
In this case, the collection of holes is naturally parametrized by
\[
	\Delta:=\{(a,b)\in\I\times \I: a< b,\ a,b\not\in\partial\I\}.
\]
If $\I=\T$, then a hole is an open interval of positive orientation from $a$ to $b$.
In this case, the interval holes are naturally parametrized by the set
\[
	\Delta:=\{(a,b)\in\I\times \I: a\neq b\}.
\]
We denote the diagonal in $\I\times\I$ by $\Delta_0:=\{(a,a): a\in\I\}$.
Observe that $\Delta_0$ is explicitly \emph{not} included in $\Delta$.
If not stated otherwise, we consider $\Delta$ equipped with the subspace topology
of the product topology on $\I\times\I$.

The geometric structure of $\Bcal_f$ is constituted 
by a configuration of vertical and horizontal segments.
Let us introduce some notation in order to describe it. 

Given a closed subset $X\subseteq\Delta$, we define 
$\Hcal(X)$ to be the family of
non-trivial maximal horizontal line segments in $X$, and $\Vcal(X)$ to be the family of
non-trivial maximal vertical line segments in $X$.
We define the set of \emph{double points} $\Dcal(X)$ to be the collection of points in $X$
which are in the intersection of an element of $\Hcal(X)$ and an element of $\Vcal(X)$.
The set of \emph{corner points} $\Ccal(X)\subset\Dcal(X)$ is given by those double
points which are endpoints of an element of $\Hcal(X)$ and of an element of $\Vcal(X)$.
Last, given $x\in\bigcup_{H\in \Hcal(X)}H$ ($x\in\bigcup_{V\in \Vcal(X)}V$) we denote the
element of $\Hcal(X)$ ($\Vcal(X)$) containing $x$ by $H_x$ ($V_x$).

Double points will play an important part in retrieving
dynamical information from the bifurcation set.
In particular, this holds for corner points $x=(a_1,a_2)\in X$ whose coordinates are \emph{links},
that is,
there is an element in $\Hcal(X)$ whose second coordinate coincides with $a_1$ and an
element in $\Vcal(X)$ whose first coordinate equals $a_2$.
We refer to such an $x$ as a \emph{step}.
Given a step $x=(a_1,a_2)\in X$, we call the maximal
collection of steps $F_x=\{\ldots, (a_1,a_2),(a_2,a_3),\ldots \}\ssq\Ccal(X)$, 
where for each element $y\in F_x$ there is a finite sequence
$y=y_{1},\ldots,y_{n}=x\in F_x$ such that $y_i$ shares a link with $y_{i+1}$ ($i=1,\ldots, n-1$),
a \emph{stair}.
Note that $F_x$ is well defined and uniquely determined by $x$.
Given $F_x=\{(a_1,a_2),\ldots,(a_{p-1},a_p)\}$ is finite and $\I=[0,1]$, we also
refer to $a_1$ and $a_p$ as \emph{terminal links}.
The \emph{length} of a stair is the cardinality of its links.
Let us point out that the above terminology originates from
the situation described in Theorem \ref{t.main A} \eqref{thmA: 2}:
for any step $x\in \Dcal(\Bcal_f)$ the segments $H_x$ and $V_x$ accumulate at
the diagonal, so that the set $\bigcup_{y\in F_x}(H_y\cup V_y)$ resembles
the shape of a stair (see also Figure~\ref{f.approx bif set doubling map}).

We can now state the first main assertion which is proven in Section \ref{s.thmA}.

\begingroup
\renewcommand\thethm{\Alph{thm}}
\begin{thm}\label{t.main A}
	Assume that $f:\I\to\I$ is continuous, transitive and not minimal.
	Then $\Bcal_f$ is closed and the following hold.
	\begin{enumerate}
		\item $\Bcal_f\neq\emptyset$ and $\Int(\Bcal_f)=\emptyset$. \label{thmA: 1}
		\item All elements of $\Hcal(\Bcal_f)$ and $\Vcal(\Bcal_f)$ accumulate
			at $\Delta_0$ and $\Bcal_f=\bigcup_{H\in \Hcal(X)}H\cup \bigcup_{V\in \Vcal(X)}V$.\label{thmA: 2}
		\item $\Dcal(\Bcal_f)$ is closed and totally disconnected.\label{thmA: 3}
		\item Each endpoint of an element of $\Hcal(\Bcal_f)$ or
			$\Vcal(\Bcal_f)$ is in $\Dcal(\Bcal_f)$.\label{thmA: 4}
		\item $\Bcal_f$ is path-connected. \label{thmA: 5}
		\item 	If $(f_n)_{n\in\N}$ is a sequence of continuous functions on $\I$ converging
			uniformly to $f$, then every accumulation point of $\Bcal_{f_n}$
			(w.r.t. the Hausdorff metric) is contained in $\Bcal_f$. \label{thmA: 6}
        \item Every stair of length $p$ in $\Bcal_f$ corresponds to a unique periodic orbit of period $p$.
        Furthermore, all but finitely many periodic orbits correspond to a stair.\label{thmA: 7}
	\end{enumerate}
\end{thm}
\endgroup

\begin{figure}[h]
	\centering
	\includegraphics[width=0.51\textwidth]{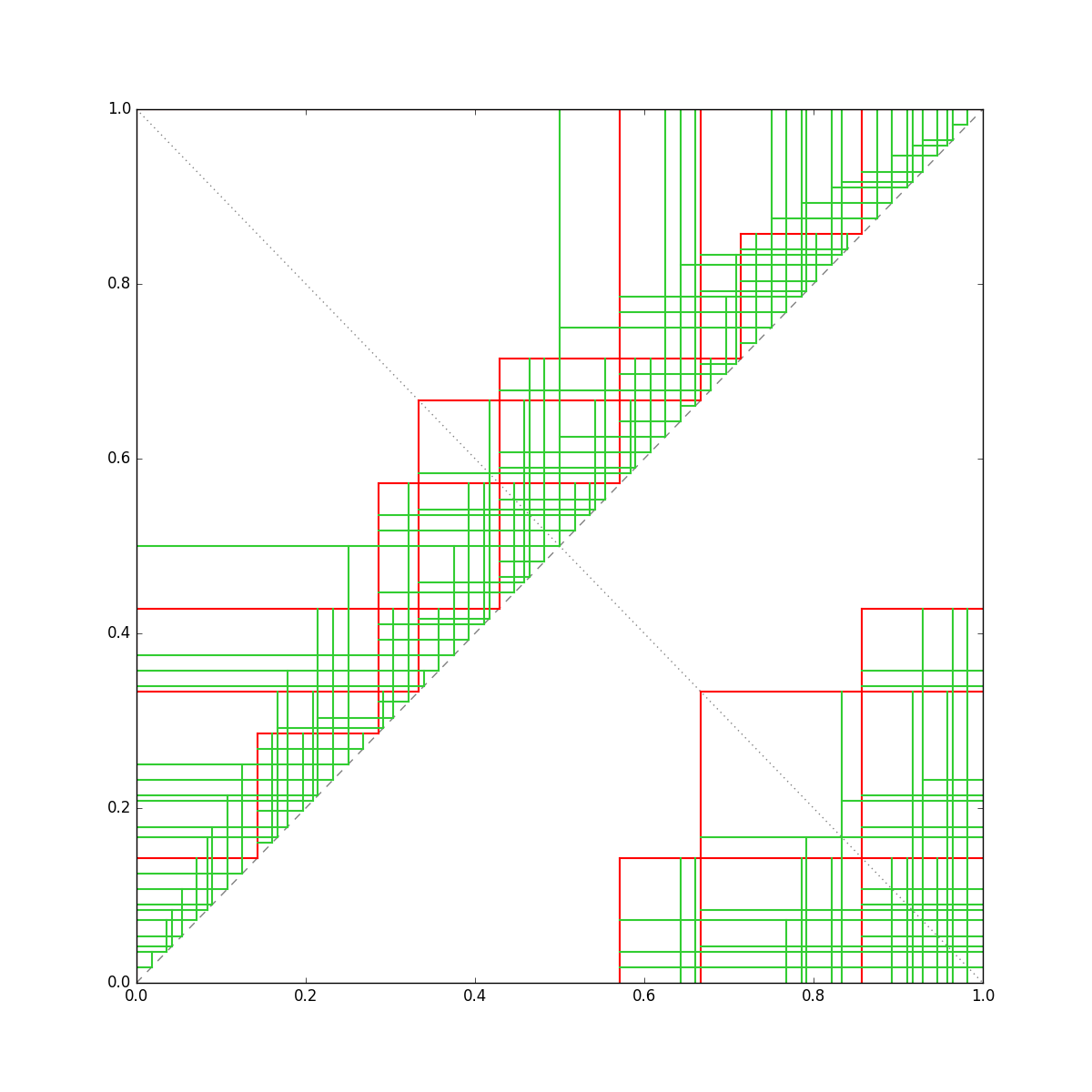}
	\hspace{-0.6cm}
	\includegraphics[width=0.51\textwidth]{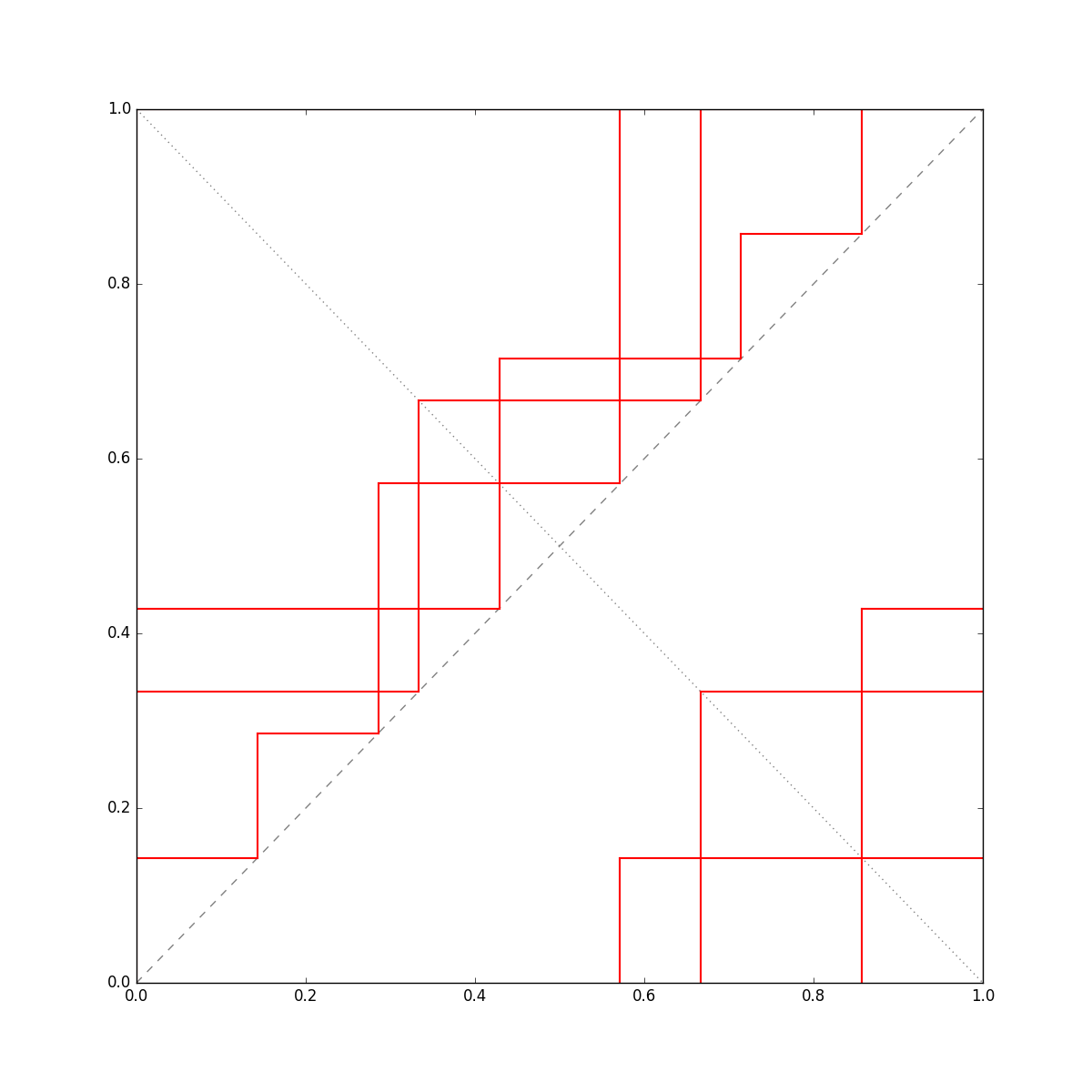}
	\caption{\small The left figure represents an approximation of the bifurcation
		set of the doubling map on the circle.
		The right figure shows the stairs of this bifurcation set corresponding
		to periodic orbits of period at most three.\normalsize}
	\label{f.approx bif set doubling map}
\end{figure}

Observe that point \eqref{thmA: 7} yields the important fact that periodic points and
their periods can be identified in the bifurcation set.
This in particular implies that the topological entropy for transitive non-minimal
piecewise monotone maps can be deduced from the bifurcation set (see
Section \ref{ss.invariants} for more details).

Notice further that the second part of point
\eqref{thmA: 2} implies that the bifurcation set is a collection of horizontal and vertical segments,
while the first part of \eqref{thmA: 2} gives --together with point \eqref{thmA: 4}--
that these segments can essentially be obtained by drawing a horizontal and vertical line from the double points
to the diagonal.\footnote{For $\I=\T$, this is true for all segments.
For $\I=[0,1]$, this is true for all
but those lines in $\Delta$ with arbitrarily small first or second coordinate, see also the previous footnote.}
This observation emphasizes the importance of double points which
is even more prominent due to their close relation to \emph{nice points} introduced in \cite{Martens1994}
(see also Remark~\ref{rem: nice points}).

As we will see, natural representatives of double points originate from the periodic points $\Per(f)$
and the preperiodic points of $f$ (see Proposition~\ref{p.endpointbifpoint} below).
It turns out that periodic and preperiodic orbits are of general importance also
besides the associated double points.
With Theorem~\ref{t.main B}, our second main result, we obtain assumptions which
guarantee that already by drawing vertical and horizontal lines from points in
the bifurcation set with one periodic and preperiodic coordinate and taking the closure of
the respective union of segments in $\Delta$ \emph{recovers} the bifurcation set.

In order to state Theorem~\ref{t.main B}, we need to introduce some further notation.
Let $x=(a,b)$ be a corner point of $\Bcal_f$.
We say that $x$ is \emph{isolated} in $\Bcal_f$ whenever for some
neighborhood $U$ of $x$ in $\Delta$ it holds
\[
	U\cap\Bcal_f=U\cap (H_x\cup V_x).
\]
Moreover, we call $x$ \emph{isolated from below} whenever for some neighborhood
$U$ of $x$ in $\Delta$ it holds for every $(a',b')\in\Bcal_f\cap U\setminus (H_x\cup V_x)$
that
\[
	a'\in\I\setminus(a,b)\quad\textnormal{or}\quad
	b'\in\I\setminus(a,b).
\]
Otherwise we call $x$ \emph{accumulated from below}.

The next statement yields the sensitivity of $\Bcal_f$ on the dynamical behavior
of the critical points $\Cri(f)$ of $f$.
We would like to remark that an essential ingredient of its proof
are shadowing and stability properties of the surviving sets (see Section~\ref{s.proofthmB}).

\begingroup
\renewcommand\thethm{\Alph{thm}}
\begin{thm}\label{t.main B}
	Suppose $f:\I\to\I$ is a continuous, transitive, not minimal and piecewise monotone map.
	Then the following hold.
	\begin{enumerate}
		\item If $\Per(f)\cap\Cri(f)=\emptyset$, then every step is isolated from below.
			Moreover, in case $\Cri(f)$ is empty or contains only transitive points, we have
			that $f$ is a continuity point\footnote{As in Theorem~\ref{t.main A} \eqref{thmA: 6},
			we consider the space of all continuous maps $f:\I\to\I$ equipped with
			the uniform topology, and the space of all non-empty closed subsets of
			$\Delta$ endowed with the Hausdorff metric.} of the bifurcation set and
			that $\Bcal_f$ can be recovered from periodic and preperiodic points.
		\item If $\Per(f)\cap\Cri(f)\neq\emptyset$, then there is at least one step
			accumulated from below or $f$ is a discontinuity point for the bifurcation set.
	\end{enumerate}
\end{thm}
\endgroup

Our last statement is an application of the above theorems and existing results concerning the family of
restricted tent maps (see Section \ref{s.restricted tent map} for the details).
The presentation here is a simplified version of Theorem~\ref{t.bifsetfamiliy}.

\begingroup
\renewcommand\thethm{\Alph{thm}}
\begin{thm}\label{t.main C}
	Let $(T_s)_{s\in[\sqrt{2},2]}$ be the family of restricted tent maps.
	Then there exist two disjoint and dense subsets of parameters, denoted by
	$\mathcal{I}$ and $\mathcal{J}$ where $\mathcal{I}$ has full measure, such that:
	\begin{enumerate}
		\item For $s\in\mathcal{I}$ every step is isolated from below and $s$ is
			a continuity point of $s\mapsto \Bcal_{T_s}$.
		\item For $s\in\mathcal{J}$ some step is accumulated from below and $s$
			is a discontinuity point of $s\mapsto\Bcal_{T_s}$.
	\end{enumerate}
\end{thm}
\endgroup

We close the introduction noting that although the bifurcation set itself is clearly
not a dynamical invariant, we can easily introduce an induced invariant, see
Section \ref{ss.invariants}.
By means of this idea, each topological property of the bifurcation set
turns into a topological invariant.
This aspect as well as the relation with periodic orbits,
topological entropy, and some measure theoretic aspects
are further explained in Section \ref{s.interpretations and invariants}.
From this discussion the following natural question arises:

\begin{center}
\emph{Which dynamical invariants of transitive non-minimal one-dimensional
	dynamics can be obtained from the bifurcation set?}
\end{center}


\section{Interpretation of $\Bcal_f$ and induced invariants}\label{s.interpretations and invariants}

We start with a formal definition of the bifurcation set.
Consider a continuous map $f:\I\to\I$.
The \emph{surviving set} of $f$ with respect to $(a,b)\in\Delta$ is defined as
\[
    \Scal_f(a,b):=\bigcap\limits_{n=0}^{\infty} f^{-n}(\I\backslash
		(a,b))=\left(\bigcup\limits_{n=0}^{\infty}f^{-n}(a,b)\right)^c.
\]
Observe that surviving sets are $f$-invariant.
Further, we define the \emph{bifurcation set} of $f$ as the following set of parameters,
$$ \Bcal_f:=\{(a,b)\in\Delta:\  a\in \Scal_f(a,b)\textnormal{ or }b\in\Scal_f(a,b)\}.$$
Note that if $(a,b)\in \Bcal_f$, both $a$ and $b$ may belong to $\Scal_f(a,b)$.

We omit the obvious proof of the next statement (which was formulated for transitive maps in Theorem~\ref{t.main A}, already).
\begin{prop}\label{p.closed}
	Let $f$ be a continuous self-map on $\I$.
	Then $\Bcal_f$ is closed in $\Delta$.
\end{prop}

We next provide an alternative characterization of $\Bcal_f$ as the complement of
the components of constant surviving sets.

\begin{prop}\label{p.constant}
    Suppose $(a,b)$ and $(a',b')$ are points in $\Delta$ belonging to the same connected component of
    $\Bcal_f^c$.
    Then $\Scal_f(a,b)=\Scal_f(a',b')$.
\end{prop}

The proof of Proposition \ref{p.constant} is a consequence of the next lemma.
In what follows we set
\[
	\Scal_f^N(a,b):=\bigcap\limits_{n=0}^{N}f^{-n}(\I\backslash (a,b))
\]
and note that $\Scal_f(a,b)=\bigcap_{N\in\N}\Scal_{f}^N(a,b)$.

\begin{lemma}\label{l.survival set inclusion}
    Suppose $(a,b)\in \Bcal_f^c$.
    Then there is $\eps>0$ and $\mc L\in \N$ such that
	\[
		\Scal_f^{N+2\mc L}(a,b)\ssq \Scal_f^{N+\mc L}(a',b')\ssq\Scal_f^N(a,b),
	\]
	for all $(a',b')\in B_{\eps}(a,b)$ and all $N\in \N$.
\end{lemma}
\begin{proof}
 Note that for $(a,b)$ as in the assumptions, there is $\eps>0$ with $B_\eps(a,b)\ssq \Bcal_f^c$ such that there are $\ell_a$ and $\ell_b$ in $\N$
 with $f^{\ell_a}(B_\eps(a)),f^{\ell_b}(B_\eps(b))\ssq (a+\eps,b-\eps)$.
 Let $\mc L:=\max \{\ell_a,\ell_b\}$.

 Consider $(a',b')\in B_\eps(a,b)$ and suppose $a\leq a'$ and $b\leq b'$ (the other cases work similarly).
 Trivially, $\Scal_f^{N+2\mc L}(a,b)\ssq\Scal_f^{N+2\mc L}(a',b)$.
 Next, observe that if $x\in (a',b')\setminus(a',b)$, then by definition of $\ell_b$, we have $f^{\ell_b}(x)\in(a',b')$ and hence
 $\Scal_f^{N+2\mc L}(a',b)\ssq\Scal_f^{N+2\mc L-\ell_b}(a',b')\ssq\Scal_f^{N+\mc L}(a',b')$.
 Likewise, we see that $\Scal_f^{N+\mc L}(a',b')\ssq \Scal_f^{N}(a,b')$ which clearly yields
 $\Scal_f^{N+\mc L}(a',b')\ssq \Scal_f^{N}(a,b)$.
 This finishes the proof.
\end{proof}

Observe that with $\eps$ and $\mc L$ as above,
we hence have for all $(a_0,b_0)$ and $(a_1,b_1)$ in $B_\eps(a,b)\ssq \Bcal_f^c$ and all $N\in \N$
that $\Scal_f^{N+4\mc L}(a_0,b_0)\ssq \Scal_f^{N+2\mc L}(a_1,b_1)\ssq\Scal_f^N(a_0,b_0)$.
This immediately yields Proposition \ref{p.constant}.

It is immediate that $\Bcal_f=\emptyset$ for $f:\I\to\I$ minimal.
Notice that Proposition \ref{p.constant} offers the converse of this
statement.\footnote{More precisely, in the interval case, Proposition \ref{p.constant}
yields transitivity for all points except $0$ and $1$.
Yet, denseness of periodic points for transitive maps implies minimality of $f$, see
the remark before Proposition \ref{prop: stairs are path-connected}.
Further, recall that there are no minimal continuous maps on $[0,1]$,
i.e., $\Bcal_f$ is always non-empty for $\I=[0,1]$.}
This also yields the first part of point \eqref{thmA: 1} of Theorem \ref{t.main A}.

\begin{cor}
 $\Bcal_f=\emptyset$ if and only if $f$ is minimal.
\end{cor}

Recall that given a probability measure $\mu$ on $\I$, the
\emph{(exponential) escape rate} of a hole $(a,b)$ with respect to $\mu$ is defined as
\[
	\rho(\mu,(a,b)):=-\lim_{N\to\infty}\frac{1}{N}\log\mu(\Scal_f^N(a,b)).
\]
If the above limit does not exist, we may likewise consider the \emph{upper} and
\emph{lower escape rate} by considering the $\limsup$ and $\liminf$, respectively.
For more information about escape rates and related concepts, see the
references at the beginning of the introduction.

Another dynamical characterization of the bifurcation set is the following which is again a consequence of
Lemma~\ref{l.survival set inclusion}.
\begin{cor}\label{c.escape rates}
	For every probability measure $\mu$ on $\I$, the lower and upper escape rate
	are constant on each connected component of the complement of $\Bcal_f$.
\end{cor}

Clearly, this result remains true when we consider non-exponential escape rates, too.

\subsection{The bifurcation set as a strict invariant and deduced invariants}\label{ss.invariants}

In the following, we discuss different dynamical invariants involved with the bifurcation set.
First, let us assume that $f$ and $g$ are \emph{conjugate}, i.e., $\pi\circ f=g\circ\pi$
where $\pi:\I\to\I$ is a homeomorphism.
Then $\Bcal_g=\{(\pi(a),\pi(b))\in\Delta\: (a,b) \in \Bcal_f\}$ if
$\pi$ is order preserving and
$\Bcal_g=\{(\pi(b),\pi(a))\in\Delta\: (a,b) \in \Bcal_f\}$ otherwise.
Hence, the bifurcation sets of conjugate maps are homeomorphic via a uniformly continuous
self-homeomorphism on $\Delta$.

Now, for subsets $X,Y\subseteq\Delta$ we can define an equivalence relation by setting $X\sim Y$
iff there is a uniformly continuous homeomorphism $p:\Delta\to\Delta$ with $p(X)=Y$.
Then the equivalence class $[\Bcal_f]$ defines a topological dynamical
invariant for $f$.
Furthermore, any topological property of $\Bcal_f$, preserved under uniformly
continuous homeomorphisms, is a dynamical invariant of $f$.
In this spirit, if we start from $[\Bcal_f]$, we can easily index the stairs and
their lengths in $\Bcal_f$.
Accordingly, in the light of point \eqref{thmA: 7} of Theorem \ref{t.main A}, we can index the periodic
orbits (all but finitely many if $\I=[0,1]$) and their periods by an inspection of $[\Bcal_f]$ for
transitive maps.

In particular, we can deduce for a transitive
non-minimal piecewise monotone map $f:\I\to\I$ that its topological entropy $h(f)$
can be recovered from $\Bcal_f$.
For this recall that a continuous map $f:\I\to\I$ is called \emph{piecewise monotone}
if there are finitely many intervals $I_1,\ldots,I_n$ in $\I$ with $\I\ssq \bigcup_{\ell=1}^n I_\ell$
such that $f$ is monotone on each ${I_\ell}$.
For this kind of maps we have that
\[
	h(f)\leq\limsup\limits_{n\to\infty}\frac{1}{n}\log\#\{x\in\I: f^n(x)=x\},
\]
see \cite[Corollaries 3 and 3']{MisiurewiczSzlenk1980}.
Moreover, in Remark \ref{r.piecewise monotone from uniformly expanding} we explain
that every transitive non-minimal piecewise monotone map is conjugate to a map with
constant slope.
This in turn implies that each monotone piece of $f$ intersects the diagonal at
most one time.
Accordingly, we get $h(f)=\limsup_{n\to\infty}1/n\log\#\{x\in\I: f^n(x)=x\}$
(see for example \cite[p.\ 218]{AlsedaLlibreMisiurewicz2000} for more details) and we
obtain the desired property of $\Bcal_f$.

Another dynamical invariant visible in $\Bcal_f$ for a continuous self-map $f$
on $\I$ is the group of automorphisms $\Aut(f)$.
These are all homeomorphisms $\pi:\I\to\I$ commuting with $f$, i.e., $f\circ\pi=\pi\circ f$.
Each $\pi\in\Aut(f)$ defines a map $\hat\pi:\Delta\to\Delta$ mapping $(a,b)$ to
$(\pi(a),\pi(b))$ or $(\pi(b),\pi(a))$ depending on whether $\pi$ is order
preserving or reversing, respectively.
Accordingly, we get that $\Bcal_f$ is invariant under $\hat\pi$ and this means
$\pi$ represent a certain symmetry of the bifurcation set.
For an example of this observation, see Figure \ref{f.approx bif set doubling map},
where the automorphism $\pi=-\mathrm{Id}$ of the doubling map is visible in the
symmetry along the off-diagonal.

Finally, from an ergodic point of view, let us briefly come back to the so-called nice points from \cite{Martens1994}
which were introduced to study possible ergodic behavior of $S$-unimodal
maps on the interval.
In particular, it is known that every $S$-unimodal map without periodic attractors
has the weak-Markov property, which implies the non-existence of positive
Lebesgue measure attracting Cantor sets.
Nice points are essential for proving this assertion and a simple inspection of
their definition shows that they can be derived from the bifurcation set.


\section{Proof of Theorem \ref{t.main A}}\label{s.thmA}

In this section, we study the topology of the bifurcation set
in general and for transitive systems in particular.
We will obtain Theorem \ref{t.main A} as a combination of several smaller
propositions and lemmas proven in this part.

\subsection{General properties of the bifurcation set}\label{sec:general prop bif set}

This section aims at a first understanding of basic topological properties of the
bifurcation set.

For the sake of completeness, let us start by
briefly recalling some standard notions from the theory of dynamical systems.
For $f:\I\to\I$ and $x\in \I$ we refer to $\Ocal(x):=\{f^n(x)\: n\in
\N_0\}$ as
the \emph{orbit} of $x$.
If $\Ocal(x)$ is finite, we call $x$ and likewise its orbit
\emph{preperiodic}.
If $f^n(x)=x$ for some $n\in \N$, then $x$ as well as its orbit are
referred to
as \emph{periodic} and we call $n$ a \emph{period} of $x$.
If $\overline{\Ocal(x)}=\I$, that is, if $\mc O(x)$ is dense in $\I$, we say $x$ is
\emph{transitive}.
We denote the collection of all periodic and transitive points of $f$ by
$\Per(f)$ and $\Tra(f)$, respectively.
If $\Tra(f)\neq\emptyset$, then we call $f$ \emph{transitive} and if
$\Tra(f)=\I$, we say $f$ is \emph{minimal}.
It is well known and easy to see that $\Tra(f)$ is dense in $\I$ (residual, in fact)
if $f$ is transitive.
Finally, we call a subset $A\ssq \I$ \emph{$f$-invariant} if $A$ is closed and if
$f(A)\ssq A$.
In case $A\ssq \I$ is $f$-invariant and if there is an $x\in A$ with $\overline{\mc O(x)}=A$,
we say $A$ is  a \emph{transitive} set.

Observe that the next statement yields point \eqref{thmA: 6} of Theorem~\ref{t.main A}.
In the following, we denote by $d$ the standard metric on $\I$ and by $d_\infty$
the \emph{supremum metric} on the space of continuous self-maps on $\I$.

\begin{prop}\label{prop: Bf is lower semi-continuous}
 Suppose $(f_n)_{n\in\N}$ is a sequence of continuous maps $f_n\:\I\to \I$ which
 converges uniformly to $f\:\I\to\I$.
 Then $\bigcap_{n\in \N}\overline{ \bigcup_{k\geq n} \Bcal_{f_k}}\ssq \Bcal_f$.
\end{prop}
\begin{proof}
 Suppose $(a,b)\notin \Bcal_f$.
 Then there is $\eps>0$ and $n,m\in \N$ with $f^n(a),f^m(b)\in (a+3\eps,b-3\eps)$.
 Choose $n_0$ sufficiently large so that $d_\infty(f_k^n,f^n),d_\infty(f_k^m,f^m)<\eps$ for
 all $k\geq n_0$.
 By the triangle inequality and continuity of $f$, there is $\delta>0$ such that
 $f_k^n(x)\in B_{2\eps}(f^n(a))$ and $f_k^m(y)\in B_{2\eps}(f^m(b))$
 if $x\in B_\delta(a)$ and $y\in B_\delta(b)$.
  We may assume without loss of generality that $\delta<\eps$.
 We have hence shown $(B_\delta(a) \times B_\delta(b))\cap \bigcup_{k\geq n_0} \Bcal_{f_k}=\emptyset$.
 Therefore, $\Bcal_f^c \ssq \left(\bigcap_{n\in \N} \overline{\bigcup_{k\geq n} \Bcal_{f_k}}\right)^c$.
\end{proof}

In the following, we say a set $V\ssq \Delta$ \emph{accumulates} at the diagonal $\Delta_0$ if
$\inf_{(a,b)\in V}d(a,b)=0$.
\begin{prop}\label{p.segment}
Let $f$ be a continuous self-map on $\I$.
For every point $x\in\Bcal_f$ there exists an element $H_x\in \Hcal(\Bcal_f)$
with $x\in H_x$ or an element $V_x \in \Vcal(\Bcal_f)$ with $x\in V_x$ which accumulates at the diagonal $\Delta_0$.

Moreover, if $f$ is transitive and $(a,b)$ is contained in an element of $\Vcal(\Bcal_f)$ ($\Hcal(\Bcal_f)$), then
$a\in \Scal_f(a,b)$ ($b\in \Scal_f(a,b)$).
In particular, each non-trivial maximal vertical or horizontal segment in $\Bcal_f$ accumulates at $\Delta_0$.
\end{prop}

\begin{proof}
For the first part, suppose $x=(a,b)\in \Bcal_f$ and assume without loss of generality that
$a\in \Scal_f(a,b)$.
Clearly, $a\in \Scal_f(a,b')$ for every $b'\in (a,b]$ which proves that there is a vertical segment in $\Bcal_f$
which accumulates at $\Delta_0$ and contains $x$.

For the second part, we may assume without loss of generality to be given an element $V\in \Vcal(\Bcal_f)$.
Denote by $\pi_2\:\Delta \to \I$ the canonical projection to the second coordinate.
Given $(a,b)\in V$, let us assume for a contradiction that $a\notin \Scal_f(a,b)$.
Then there is $n\in \N$ such that $f^n(a)\in (a,b)$.
Now, there clearly is a transitive point $c\in \pi_2(V)$ with $c\in (f^n(a),b)$ or
$b\in (f^n(a),c)$ and which --as its orbit is dense and thus hits
$(a,c)$-- is not in $\Scal_f(a,c)$.
Therefore, $(a,c)\notin \Bcal_f$ contradicting the assumption that $V\ssq \Bcal_f$.
This proves the statement.
\end{proof}

\begin{rem}\label{rem: nice points}
 Observe that the previous statement implies that if $f$ is transitive, we have that $a,b\in \Scal_f(a,b)$ if and
only if $(a,b)\in \mc D(\Bcal_f)$.
\end{rem}

\begin{cor}\label{c.structureBif}
Let $f$ be a continuous transitive self-map on $\I$.
Then $\bigcup_{V\in \Vcal(\Bcal_f)} V$ and $\bigcup_{H\in \Hcal(\Bcal_f)} H$
(and therefore $\Dcal(\Bcal_f)$) are closed.
\end{cor}
\begin{proof}
 Let $(a_n,b_n)_{n\in\N}$ be a sequence of points in $\bigcup_{V\in \Vcal(\Bcal_f)} V$
 (the case of $\bigcup_{H\in \Hcal(\Bcal_f)} H$ works
 similarly) converging to some $(a,b)\in \Delta$.
 By Proposition~\ref{p.segment}, we know $(a_n,b_n)$ is contained in a vertical segment which accumulates at $\Delta_0$.
 Hence, for each $b'\in (a,b]$ we have a sequence $(a_n,b_n')_{n\in\N}$ in
 $\bigcup_{V\in \Vcal(\Bcal_f)} V$ with $(a_n,b_n')\to (a,b')$ as $n\to\infty$.
 Since $\Bcal_f$ is closed (by Proposition~\ref{p.closed}), we get $\{(a,b')\: b'\in (a,b]\}\ssq \Bcal_f$, i.e.,
 $(a,b)\in \bigcup_{V\in \Vcal(\Bcal_f)} V$ which
 finishes the proof.
\end{proof}

\begin{prop}\label{prop: transitivity implies empty interior of bif-set}
 If $f\:\I\to\I$ is transitive, then $\Int(\Bcal_f)=\emptyset$.
\end{prop}
\begin{proof}
 Given $(a,b)\in \Bcal_f$, we find arbitrarily close $(a',b')$ such that
 $a'$ and $b'$ are transitive points and hence $a',b'\notin \Scal_f(a',b')$.
\end{proof}

Note that transitivity is not necessary in order to have $\Int(\Bcal_f)=\emptyset$.
For example, on $\I=[0,1]$, we may consider
\begin{align*}
 f(x):=
 \begin{cases}
  1/2-3\cdot|x-1/6| & \text{ for } 0\leq x \leq 1/3,\\
  3\cdot(x-1/3) & \text{ for } 1/3\leq x \leq 2/3,\\
  1/2+3\cdot|x-5/6| & \text{ for } 2/3\leq x \leq 1.
 \end{cases}
\end{align*}
Here, $[0,1/2]$ and $[1/2,1]$ are transitive $f$-invariant subsets and we see, similarly as in the proof of
Proposition~\ref{prop: transitivity implies empty interior of bif-set}, that $\Int(\Bcal_f)=\emptyset$.

Recall that the set of \emph{non-wandering points} of $f$ is defined by
\[NW(f):=\{x\in \I\: \forall \eps>0 \, \exists n\in\N \text{ such that } f^n(B_\eps(x))\cap B_\eps(x) \neq \emptyset\}.\]
We straightforwardly obtain the following converse of Proposition~\ref{prop: transitivity implies empty interior of bif-set}.

\begin{prop}
 Let $f$ be a continuous self-map on $\I$.
 If $\Int(\Bcal_f)=\emptyset$, then $NW(f)=\I$.
\end{prop}
Clearly, $NW(f)=\I$ is not sufficient in order to have $\Int(\Bcal_f)=\emptyset$
as can be seen by considering the identity, for example.

\subsection{Transitive case}
The statements of the previous section suggest that the additional assumption of transitivity
allows for a substantially more detailed description of the bifurcation set.
With this observation in mind, we are now taking a closer look at the transitive case.

\begin{lemma}
 If $f\:\I\to\I$ is continuous and transitive, then $\Dcal(\Bcal_f)$ is totally disconnected.
\end{lemma}
\begin{proof}
 Observe that since $\Dcal(\Bcal_f)$ is locally compact (see Corollary~\ref{c.structureBif}),
 $\Dcal(\Bcal_f)$ is totally disconnected if and only if it is zero dimensional.
 For a contradiction, we assume that $\Dcal(\Bcal_f)$ is not zero dimensional
 so that there is $(a,b)\in \Dcal(\Bcal_f)$ such that $(a,b)$ does not have arbitrarily
 small clopen neighborhoods in $\Dcal(\Bcal_f)$.
 Then there is $\eps_0>0$ such that for all $\eps\in[0,\eps_0]$ we have that the boundary of the rectangle
 $[a-\eps,a+\eps]\times [b-\eps,b+\eps]$ intersects $\Dcal(\Bcal_f)$ (note that we may assume without loss of generality
 that $\eps_0<1/2\cdot d(a,b)$).
 Observe that if $\Dcal(\Bcal_f)$ intersects one of the vertical sides of this boundary, this gives
 $v_\eps\ssq \Bcal_f$ or $v^\eps \ssq \Bcal_f$, where $v_\eps$ and $v^\eps$ are the vertical line segments
 $v_\eps = \{a-\eps\}\times (a-\eps,b-\eps_0]$ and $v^\eps = \{a+\eps\}\times (a+\eps,b-\eps_0]$, respectively.
 Likewise, if $\Dcal(\Bcal_f)$ intersects one of the horizontal sides, this implies
 $h_\eps\ssq \Bcal_f$ or $h^\eps \ssq \Bcal_f$, where
 $h_\eps = [a+\eps_0,b-\eps)\times \{b-\eps\}$ and $h^\eps = [a+\eps_0,b+\eps)\times \{b+\eps\}$.
 Hence,
 \[
	[0,\eps_0]=\bigcup_{\substack{\eps \in [0,\eps_0]\\v_\eps\ssq \Bcal_f}} \eps \cup
			\bigcup_{\substack{\eps \in [0,\eps_0]\\v^\eps\ssq \Bcal_f}}\eps\cup
			\bigcup_{\substack{\eps \in [0,\eps_0]\\h_\eps\ssq \Bcal_f}}\eps \cup
			\bigcup_{\substack{\eps \in [0,\eps_0]\\h^\eps\ssq \Bcal_f}}\eps.
 \]
 According to Corollary~\ref{c.structureBif}, the sets
 \[
	\bigcup_{\substack{\eps \in [0,\eps_0]\\v_\eps\ssq \Bcal_f}} \eps,
	\bigcup_{\substack{\eps \in [0,\eps_0]\\v^\eps\ssq \Bcal_f}} \eps,\ldots
 \]
 are closed.
 Hence by Baire's category theorem, we may assume without loss of generality that there is a non-degenerate interval
 $I\ssq\bigcup_{\substack{\eps \in [0,\eps_0]\\v_\eps\ssq \Bcal_f}} \eps$.
 But then $(a-I)\times (a,b-\eps_0]\ssq \Bcal_f$ so that $\textrm{int}(\Bcal_f)\neq \emptyset$,
 contradicting Proposition~\ref{prop: transitivity implies empty interior of bif-set}.
 \end{proof}
 Together with Corollary~\ref{c.structureBif}, the previous statement proves point (\ref{thmA: 3}) of Theorem~\ref{t.main A}.
 We next consider point (\ref{thmA: 4}).

\begin{prop}\label{p.endpointbifpoint}
Suppose $f\:\I\to\I$ is continuous and transitive.
If $(a,b)$ is an endpoint of an element of $\Hcal(\Bcal_f)$,
then $(a,b)\in\Dcal(\Bcal_f)$ and the orbit of $b$ comes arbitrarily close to $a$.
Likewise, if $(a,b)$ is an endpoint of an element of $\Vcal(\Bcal_f)$,
then $(a,b)\in\Dcal(\Bcal_f)$ and the orbit of $a$ comes arbitrarily close to $b$.
\end{prop}

\begin{proof}
We only consider $(a,b)\in \Hcal(\Bcal_f)$, the other case is similar.
By the second part of Proposition~\ref{p.segment},
we have to show that $a\in \Scal_f(a,b)$.
Since $(a,b)$ is an endpoint of a maximal horizontal segment,
the set $\{x\in \I\setminus[a,b]\colon (x,b)\in\Bcal_f^c\}$ accumulates at $a$.
By definition, for all $x$ in the previous set,
there are positive integers $n_x$ and $m_x$, so that $f^{n_x}(b)\in (x,b)$ and $f^{m_x}(x)\in(x,b)$.
As $b\in \Scal_f(a,b)$ (see Proposition~\ref{p.segment}), this gives $f^{n_x}(b)\in (x,a]$ and hence
$\inf_{n\in\N} d(f^n(b),a)=0$.
Thus, if there was $n\in \N$ with $f^n(a)\in(a,b)$ we would have that $f^m(b)\in(a,b)$ for some $m\in \N$ contradicting
the fact that $b\in \Scal_f(a,b)$.
Therefore, $a\in\Scal_f(a,b)$.
\end{proof}

Recall the definition of steps, links and stairs from the introduction.
Given a transitive self-map on $\I$, it is easy to see that if $\{ x_1,x_2, \ldots, x_p\}$ is a periodic orbit,
then each pair of adjacent points $(x_{i_0},x_{i_1})$ (where the interval $(x_{i_0},x_{i_1})$ does not intersect
the respective orbit) with $x_{i_0},x_{i_1}\notin \partial \I$
is a step and all elements in $\{ x_1,x_2, \ldots, x_p\}\setminus \partial \I$ are links.
In this way, each periodic orbit with at least two elements not contained in $\partial \I$ is naturally
associated to a stair in $\Bcal_f$.
In fact, we have the following

\begin{prop}\label{prop: stairs and periodic orbits}
Given $f:\I\to\I$ is continuous and transitive, every stair of $\Bcal_f$ is of
finite length and realized by a unique periodic orbit.
\end{prop}
\begin{proof}
Assume we are given a stair of length $p \in \N\cup\{\infty\}$.
By definition, each element $(x_i,x_{i+1})$ of the stair is a corner point so that $x_i,x_{i+1}\in \Scal_f(x_i,x_{i+1})$,
due to Proposition~\ref{p.endpointbifpoint}.
Further, as $x_i$ is a link, there is an element of $\Hcal(\Bcal_f)$ which accumulates at $(x_i,x_i)\in \Delta_0$, so that
Proposition~\ref{p.segment} yields that $x_i\in \Scal_f(c,x_i)$ for some $c<x_i$.
Hence, the orbit of $x_i$ does not hit the set $(c,x_i)\cup(x_i,x_{i+1})$ and can
therefore not accumulate at $x_{i}$.
Likewise, we obtain that the orbit of $x_{i+1}$ cannot accumulate at $x_{i+1}$.
However, due to Proposition~\ref{p.endpointbifpoint}, the orbit of $x_i$ comes
arbitrarily close to $x_{i+1}$ and the orbit of $x_{i+1}$ comes arbitrarily close
to $x_i$.
This clearly yields that $x_i$ is an iterate of $x_{i+1}$ and vice versa.
Hence, $x_i$ and $x_{i+1}$ are elements of a periodic orbit.
We conclude that all links associated to a stair come
from one and the same periodic orbit of period not bigger than $p+2$.
This proves the statement.
\end{proof}

\begin{cor}
Let $f\: \I\to \I$ be continuous and transitive.
Then, for all but finitely many $p\geq 2$, there is a one-to-one correspondence between periodic orbits of minimal period $p$
and stairs of length $p$.
\end{cor}
\begin{proof}
 By the above, there is a one-to-one correspondence between stairs and periodic orbits which contain at least two
 elements within $\I\setminus \partial \I$.
 Further, unless a given periodic orbit hits $\partial \I$, its period obviously coincides with
 the length of the associated stair.
 As there are at most two periodic orbits which hit $\partial \I$, the statement follows.
\end{proof}
\begin{rem}
 We would like to stress that in case of $\I=\T$, it is straightforward to see that the above
 one-to-one correspondence holds true for all
 periods $p\geq 2$, in fact.
\end{rem}

Slightly abusing notation, given a step $x$,
we may also refer to the point-set $S_{x}=V_{x}\cup H_{x}\ssq \Bcal_f$
as a \emph{step}.
In a similar fashion, given a stair $F_x$, we may also refer to the
union of all maximal vertical and horizontal segments
whose first and second coordinate, respectively, coincides with a link of $F_x$
as the stair $F_x$.
Notice that for $\I=[0,1]$, this union not only includes
all respective steps (considered as point-sets)
but also the horizontal and vertical segments associated to terminal links.
We may refer to these segments as \emph{terminal segments} of $F_x$.
Observe that since each stair is realized by a periodic orbit, the terminal segments accumulate at
$\{0\}\times \I$ and $\I\times \{1\}$.

By a \emph{path} in $\Bcal_f$, we refer to a continuous map $\gamma\: [0,1]\to \Delta$ with
$\gamma([0,1])\ssq \Bcal_f$.
Recall that $\Bcal_f$ is \emph{path-connected} if for all $x,y\in \Bcal_f$ there is a path $\gamma$ in $\Bcal_f$
from $x$ to $y$, that is, $\gamma(0)=x$ and $\gamma(1)=y$.
In order to prove the path-connectedness of $\Bcal_f$, we make use of the following observation
whose proof is based on the classical fact that a continuous transitive and non-minimal self-map
on $\I$ has a dense set of periodic points (for interval maps, see \cite{Sharkovsky1964} and
also \cite[Lemma 41 on p. 156]{BlockCoppel1992}; for maps on the circle, this follows from
\cite[Theorem A]{CovenMulvey1986} together with \cite[Corollary 2]{AuslanderKatznelson1979}).

\begin{prop}\label{prop: stairs are path-connected}
 Suppose $f\:\I\to\I$ is continuous and transitive.
 Given two points $(a,b)$ and $(a',b')$ on a stair $F_x$ (considered as the above union of segments),
 there is a continuous path in $\Bcal_f$ from $(a,b)$ to $(a',b')$.
\end{prop}
\begin{proof}
 We may assume without loss of generality that $(a,b)$ and $(a',b')$ lie on neighboring
 steps, that is, $(a,b)\in S_{(y_1,y_2)}$ and $(a',b')\in S_{(y_2,y_3)}$ for
 some $(y_1,y_2),(y_2,y_3)\in F_x$ (note that if $(a,b)$ or $(a',b')$ lies on a terminal segment, the following proof
 works exactly the same).
 As $f$ is transitive, there is a transitive point $y\in (y_1,y_2)$.
 By transitivity of $y$, there is $n\in \N$ such that $f^n(y)\in (y_2,y_3)$.
 Clearly, for a small enough interval $J\ssq (y_1,y_2)$ containing $y$,
 we have $f^n(J)\ssq (y_2,y_3)$.
 By denseness of periodic points, there is a periodic point $z\in J$.
 Let $z_1$ and $z_2$ be those points in the orbit of $z$ which are the furthest
 to the right in $\mc O(z)\cap (y_1,y_2)$ and the furthest to the left in
 $\mc O(z)\cap(y_2,y_3)$, respectively.
 Clearly, $(z_1,z_2)$ is a step and $S_{(z_1,z_2)}$ intersects both $S_{(y_1,y_2)}$ and $S_{(y_2,y_3)}$.
 Let $\gamma_1$ be some path in $S_{(y_1,y_2)}$ from $(a,b)$ to the unique intersection point
 $(c,d)$ of $S_{(y_1,y_2)}$ and $S_{(z_1,z_2)}$; let $\gamma_2$ be a path in $S_{(z_{1},z_{2})}$
 from $(c,d)$ to the unique intersection point
 $(c',d')$ of $S_{(z_1,z_2)}$ and $S_{(y_2,y_3)}$; let $\gamma_3$ be a path in
 $S_{(y_2,y_3)}$ from $(c',d')$ to $(a',b')$.
 Clearly, the concatenation of $\gamma_1$, $\gamma_2$ and $\gamma_3$ is a path in $\Bcal_f$ from $(a,b)$ to $(a',b')$.
\end{proof}

\begin{figure}[h]
	\centering
	\includegraphics[width=0.71\textwidth]{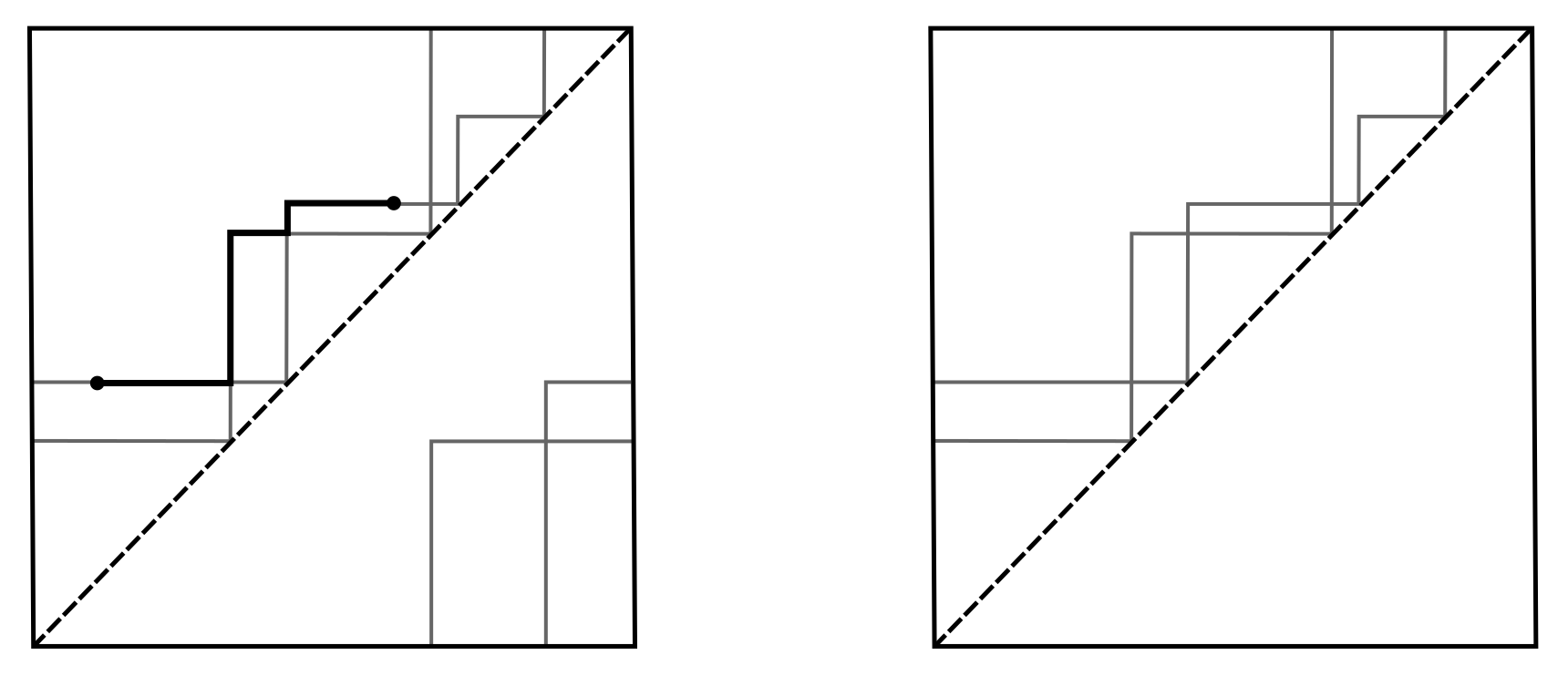}
	\caption{\small Stairs illustrated for maps on the circle (left) and on the
		interval (right). The left figure also depicts a path (bold
		line) as discussed in Proposition \ref{prop: stairs are path-connected}.\normalsize}
	\label{f.stairs and path}
\end{figure}

We next obtain point (\ref{thmA: 5}) of Theorem~\ref{t.main A}.

\begin{lemma}
 If $f\:\I\to\I$ is continuous and transitive, then $\Bcal_f$ is path-connected.
\end{lemma}
\begin{proof}
We first observe that given two points $x$ and $y$ on stairs $F_x$ and $F_y$, respectively,
there is a path in $\Bcal_f$ from $x$ to $y$.
To see this, it suffices --due to the previous statement-- to show that there is a non-empty intersection between
some segment associated to $F_x$ and some segment associated to $F_y$.
This, however, follows immediately from the fact that
on $\I=\T^1$, each stair wraps around $\Delta_0$ while on $\I=[0,1]$,
the horizontal and vertical terminal segment of each stair accumulates at
$\{0\}\times \I$ and $\I\times \{1\}$, respectively (see Figure \ref{f.stairs and path}).

Now, suppose we are given arbitrary points $x,y\in \Bcal_f$.
Due to Proposition~\ref{p.segment}, we may assume without loss of generality that
$x=(a,b)$ lies on a non-trivial horizontal segment $H$.
Due to the denseness of periodic points, we find a periodic point $c\in \I$ with
$c\in (a,b)$.
Without loss of generality, we may assume that $\mc O(c)$ contains at least two points in $\I\setminus\partial \I$.
Choose $c'$ to be the right-most point in $\mc O(c)\cap (a,b)$.
Then, the vertical segment (terminal or not) of the stair associated to $\mc O(c)$ which accumulates at $(c',c')$ clearly
intersects $H$.
Hence, there is a path $\gamma_1$ from $x$ to a point $z_x$ on a stair $F_{z_x}$ in $\Bcal_f$.
Likewise, we obtain a path $\gamma_2$ from $y$ to a point $z_y$ on a stair $F_{z_y}$ in $\Bcal_f$ whose inverse (from $z_y$
to $y$) we denote by $\overline{\gamma_2}$.
By the above observation, there is a path $\gamma_3$ in $\Bcal_f$ from $z_x$ to $z_y$.
Altogether, the concatenation $\overline{\gamma_2}\cdot \gamma_3\cdot \gamma_1$ is a path in $\Bcal_f$ from $x$ to $y$ which proves the statement.
\end{proof}


\section{Proof of Theorem \ref{t.main B}}\label{s.proofthmB}

In this section, we turn to the problem of identifying critical points and their
dynamical behavior by means of the bifurcation set.
For this recall that given a continuous map $f\:\I\to\I$, a point $x\in \I$ is
referred to as \emph{critical} if there is no neighborhood of $x$ on which $f$ is
monotone.
The collection of all critical points of $f$ is denoted by $\Cri(f)$.

Let us point out that Theorem~\ref{t.main B} follows from
Theorem \ref{t.bifsetandcrit} (the main result of this section), 
see Remark \ref{r.piecewise monotone from uniformly expanding}.

\subsection{Implications of hyperbolicity}

Besides transitivity, we will impose additional assumptions on the map $f$.
In particular, we will assume certain forms of hyperbolicity. As we are dealing with
results of a topological flavor, we consider the following definition of hyperbolicity:
an $f$-invariant set $A\ssq \I$ is referred to as \emph{hyperbolic} for a continuous map $f:\I\to\I$
and a compatible metric $d$ if there exist $\eps>0$ and $\lambda>1$ and an open neighborhood $U$ of $A$ such that
$d(f(x),f(y))>\lambda \cdot d(x,y)$ for all $x,y\in U$ with $d(x,y)<\eps$.
In this case, we may also say that $f$ is \emph{$\eps$-locally $\lambda$-expanding on $U$} (with respect to $d$).
Note that a smooth map $f:\I\to\I$ which is hyperbolic on an invariant set $A$ in the classical sense is also hyperbolic
in the above sense with respect to some metric $d$ equivalent to the usual one
(see for instance the proof of Theorem~2.3 in Chapter~III of \cite{MeloStrien1993}).
Henceforth, all metrics are considered to be equivalent to the standard metric on $\I$ and throughout denoted by $d$.

We call $x\in \I$ \emph{hyperbolic} if $\overline{\Ocal(x)}$ is
hyperbolic in the above sense.
Notions like \emph{hyperbolic steps} or \emph{hyperbolic double points} are defined in the natural way.

Suppose $x\in\I$ is a periodic point of $f:\I\to\I$ with minimal period $p$.
We say that $f$ \emph{preserves orientation at} $a\in\Ocal(x)$ whenever
$f^p|_J$ preserves orientation in some neighborhood $J$ of $a$.
Otherwise, we say that $f$ \emph{reverses orientation at} $a$.
Given $a,b\in\Ocal(x)$, we denote by $n_{a,b}$ the minimum time for going from
$a$ to $b$ by iteration of $f$.
We say that $f$ \emph{preserves orientation from $a$ to $b$} whenever
$f^{n_{a,b}}|_J$ preserves orientation in some neighborhood $J$ of $a$.
Otherwise, we say that $f$ \emph{reverses orientation from $a$ to $b$}.

Concerning the next statement, recall that due to Proposition~\ref{prop: stairs and periodic orbits}
every step is associated to a periodic point.
We may hence refer to the period of this periodic point also as the \emph{period} of the respective step.

\begin{lemma}\label{l.periodiccorner}
	Let $f:\I\to\I$ be continuous and transitive.
	Suppose $(a,b)\in\Bcal_f$ is a hyperbolic step of period $p$.
	The following holds.
	\begin{enumerate}
		\item If $f$ reverses orientation at $a$ or $b$, then $(a,b)$ is an isolated
			corner point of $\mathcal{B}_f$.\label{lem: hyperbolic steps item 1}
		\item If $f$ preserves orientation both at $a$ and $b$, then
		     $(a,b)$ is isolated from below.\label{lem: hyperbolic steps item 2}
	\end{enumerate}
\end{lemma}
\begin{proof}
 We start by proving \eqref{lem: hyperbolic steps item 1}.
 For a contradiction, suppose there is $(a',b')\in \Bcal_f\setminus (V_{(a,b)}\cup H_{(a,b)})$ arbitrarily close to $(a,b)$.
 Without loss of generality, we may assume $a'\in \Scal_f(a',b')$.

 First, consider $a'=a$.
 Then we necessarily have $b'\in [a,b]^c$ (since otherwise we had $(a',b')\in V_{(a,b)}$)
 and thus $(a',b')\supsetneq (a',b)$.
 As the orbit of $a'=a$ accumulates at $b$ (see Proposition~\ref{p.endpointbifpoint}), this gives $a'\notin \Scal_f(a',b')$.

 Now, consider $a'\neq a$ and assume without loss of generality that $a'\in (a,b)$ (the other case can be dealt with similarly).
 Assuming that $(a',b')$ is sufficiently close to $(a,b)$,
 we have $f^{2p}(a') \in (a',b')$ since $f$ is expanding in a neighborhood of $\mc O(a)$
 (as it is hyperbolic on $\mc O(a)$) and $f^{2p}$ is order-preserving in a
 neighborhood of $a$.
 Hence, $a'\notin \Scal_f(a',b')$.
 This contradicts the assumptions on $a'$ and finishes the proof of the first part.

 Let us now turn to part \eqref{lem: hyperbolic steps item 2}.
 Assume for a contradiction that there is $(a',b')\in \Bcal_f$
 arbitrarily close to $(a,b)$ with $[a',b']\ssq (a,b)$.
 As $f^p$ is expanding and order preserving both in $a$ and $b$, we have
 $f^p(a'),f^p(b')\in (a',b')$ if $a'$ and $b'$ are sufficiently close to $a$ and $b$.
 Hence, $a',b'\notin \Scal_f(a',b')$ which contradicts the assumptions.
\end{proof}
\begin{rem}
 Assume the situation of the previous statement.
 It is not hard to see that if $f$ preserves orientation at $a$ and $b$ and
 additionally preserves orientation from $a$ to $b$, then $(a,b)$ is actually isolated.
 In particular, if $f$ is uniformly expanding, every step is isolated.
\end{rem}

Recall that given $x\in \I$, its \emph{$\w$-limit set} $\w_f(x)$ is defined to be the collection
of all accumulation points of $\mc O(x)$.
It is well known and easy to see that $\w_f(x)$ is non-empty, compact and contains \emph{recurrent points}, that is,
there is $y\in \w_f(x)$ such that $y\in \w_f(y)$.

We call a double point $(a,b)\in\Bcal_f$ \emph{(pre)periodic}, if both $a$ and $b$
are (pre)periodic.
The proof of the next statement makes use of standard shadowing arguments.

\begin{lemma}\label{lem: approximation by preperiodic elements}
 Let $f\:\I\to\I$ be continuous and transitive.
 Consider $(a,b)\in \Bcal_f$ with $a\in \Scal_f(a,b)$ and $b\notin \Scal_f(a,b)$
 and suppose the orbit of $a$ is hyperbolic.
 If $a$ is not preperiodic, then $(a,b)$ is accumulated by points of the form $(\tilde a,b)\in \Bcal_f$ with
 $\tilde a$ preperiodic,
 $\mc O(\tilde a)$ hyperbolic and $\tilde a \in \Scal_f(\tilde a,b)$.
 A similar statement holds if we interchange the roles of $a$ and $b$.

 Moreover, if $(a,b)\in \Bcal_f$ is a double point which is not preperiodic and the orbits of $a$ and $b$ are
 hyperbolic, then $(a,b)$ is accumulated by hyperbolic preperiodic double points.
\end{lemma}
\begin{proof}
 Let $a\in \Scal_f(a,b)$ (the other case is similar) and assume $a$ is not preperiodic.
 Due to the assumptions, there is an open set $U$ (and a compatible metric $d$) with $\mc O(a)\cup \w_f(a)\ssq U$ such
 that $f$ is $\delta$-locally $\lambda$-expanding on $U$.
 Without loss of generality, we may assume that $\delta>0$ is such that $B_\delta(x)\ssq U$
 for all $x\in\mc O(a)\cup \w_f(a)$.

 Choose some $\eps<\delta/2$ and let $c\in \w_f(a)$ be a recurrent point.
 Pick $n\in \N$ with $d(f^n(c),c)<\eps$.
 Since $f$ is $\delta$-locally $\lambda$-expanding on $U$, we may assume without loss of generality that $n$ is
 large enough to ensure that $f^n(B_\eps(c))\supseteq B_{2\eps}(f^n(c))$.
 Choose $I$ to be the connected component of $\bigcap_{\ell=1}^nf^{-\ell}\left(B_{2\eps}(f^{\ell}(c))\right)\cap B_{\eps}(c)$
 which contains $c$.

 By the assumptions on $n$, we have $f^n(I)=B_{2\eps}(f^n(c))\supseteq B_{\eps}(c)\supseteq I$.
 Hence, there is a periodic point $d\in I$ of period $n$ whose orbit is $2\eps$-close to $\w_f(a)$ (by definition of $I$).
 Since $f$ is $\delta$-locally $\lambda$-expanding on $U$, there further is $m\in \N$ and a point $a'\in \I$ such that $f^m(a')=d$ and
 $\max_{\ell=0,\ldots,m-1} d(f^\ell(a'),f^\ell(a))<2\eps$.
 Set $\tilde a$ to be the right-most point of $\mc O(a')\cap B_{2\eps}(a)$.

 Let $b\notin \Scal_f(a,b)$.
 Then $\mc O(a)$ is at positive distance to $b$ (otherwise $a$ would not survive) and we may assume $\eps>0$ to be small enough
 to ensure that $a$ does not come $2\eps$-close to $b$ so that
 $\mc O(\tilde a)\cap (\tilde a,b)=\emptyset$, i.e., $\tilde a\in \Scal_f(\tilde a,b)$.
 As $\eps$ can be chosen arbitrarily small, the first part follows.

 Next, let us assume $b\in \Scal_f(a,b)$, that is, $(a,b)$ is a double point.
 If $\mc O(\tilde a)\cap B_{2\eps}(b)\neq \emptyset$, set $\tilde b$ to be the left-most point
 in $\mc O(\tilde a)\cap B_{2\eps}(b)$.
 Then, $(\tilde a,\tilde b)$ is preperiodic and moreover a double point
 with $d(\tilde a,a),d(\tilde b,b)<2\eps$.
 If $\mc O(\tilde a)\cap B_{2\eps}(b)=\emptyset$, then $(\tilde a,b)$ is a double point.
 If $b$ is preperiodic, $(\tilde a,b)$ is hence a preperiodic double point $2\eps$-close
 to $(a,b)$.
 If $b$ is not preperiodic and $\mc O(\tilde a)\cap B_{2\eps}(b)= \emptyset$, repeat the above argument
 for $(\tilde a,b)$ with the roles of $a$ and $b$ interchanged.
 In all cases, we end up with a preperiodic double point $(\tilde a,\tilde b)$ with
 $d(\tilde a,a),d(\tilde b,b)<4\eps$.
 Since $\eps>0$ can be chosen arbitrarily small, this finishes the proof.
\end{proof}

Given $(a,b)\in \Bcal_f$, we call $\{a,b\}\cap \Scal_f(a,b)$ the \emph{surviving endpoints} of $(a,b)$.

\begin{thm}
 Let $f\:\I\to\I$ be continuous and transitive and suppose every critical point of $f$ is transitive.
 Assume further that every transitive invariant subset of $\I$ which does not contain a critical point is hyperbolic.
 Then the mapping $g\mapsto \Bcal_g\in 2^{\Delta}$ is continuous in $f$ with
 respect to the uniform topology on the space of continuous self-maps on $\I$ and the Hausdorff metric on $2^{\Delta}$.
\end{thm}
\begin{proof}
 According to Proposition~\ref{prop: Bf is lower semi-continuous}, given a sequence
 $(f_n)_{n\in\N}$ of continuous maps $f_n\:\I\to\I$ with $f_n\to f$ uniformly as $n\to\infty$,
 it suffices to show that for each  $\eps>0$ there is $n_0$ such that
 for all $n\geq n_0$ we have $B_{\eps}\left(\Bcal_{f_n}\right)\supseteq \Bcal_f$.

 Pick $\eps>0$.
 Observe that due to Proposition~\ref{p.closed}, $\Bcal_f$ is precompact so that
 there are finitely many $(a_1,b_1),\ldots,(a_M,b_M)\in \Bcal_f$ with
 $\Bcal_f\ssq \bigcup_{j=1}^M B_{\eps}\left((a_j,b_j)\right)$.
 As the elements of $\Cri(f)$ are transitive, we further have that
 the orbits of the surviving endpoints of $(a_1,b_1),\ldots,(a_M,b_M)$ are
 at a positive distance to $\Cri(f)$.
 By the assumptions, this yields that the orbits of the surviving endpoints are hyperbolic.
 Hence, by Lemma~\ref{lem: approximation by preperiodic elements}, there are
 $(\tilde a_1,\tilde b_1),\ldots,(\tilde a_M,\tilde b_M)\in \Bcal_f$
 with $\Bcal_f\ssq \bigcup_{j=1}^M B_{2\eps}\left((\tilde a_j,\tilde b_j)\right)$ and
 such that at least one of the surviving endpoints of each pair among
 $(\tilde a_1,\tilde b_1),\ldots,(\tilde a_M,\tilde b_M)$ is preperiodic and hyperbolic.
 We denote these surviving endpoints by
 $y_1,\ldots,y_N$ (where $M\leq N\leq 2M$).

 Let $p$ be bigger than $\max_{\ell=1,\ldots, N} \#\mc O(y_\ell)$ and such that
 $y_\ell$ ($\ell=1,\ldots,N$) is eventually $p$-periodic.
 Observe that $f^p$ maps the points $y_1,\ldots,y_N$ to
 fixed points of $f^p$.
 By possibly going over to multiples of $p$, we may assume without loss of generality that there is $\delta>0$ such that
 $f^p$ is $(2\delta)$-locally $3$-expanding in a neighborhood of $\bigcup_{\ell=1}^N \mc O(y_\ell)$.
 We may further assume $\delta$ to be small enough such that
 \[
	d(f^m(x),f^m(y_\ell))<1/2\cdot\min\left\{ \min_{j=1}^M d(\tilde a_j,\tilde b_j),
	\eps\right\}=:\eps_0 \qquad (m=0,\ldots,2p)
 \]
 whenever $d(x,y_\ell)<\delta$ with $\ell\in\{1,\ldots,N\}$.
 Choose $n_0$ such that for all $k\geq n_0$ we have
 $d_\infty(f^m_k,f^m)<\delta$ for all $m=1,\ldots,p$.

 Now,
 $f^p(B_{\delta}(y_\ell))\supseteq B_{3\delta}(f^p(y_\ell))$ so that
 $f^p_k(B_{\delta}(y_\ell))\supseteq B_{2\delta}(f^p(y_\ell))$ for $\ell=1,\ldots,N$ and $k\geq n_0$.
 Similarly, we have
 $f^{p}_k(B_{\delta}(f^p(y_\ell)))\supseteq B_{2\delta}(f^p(y_\ell))$.
 Altogether, this shows that for all $k\geq n_0$ there is
 $x_\ell^k\in B_{\delta}(y_\ell)$ with $f^p_k(x_\ell^k)\in B_\delta(f^p(y_\ell))$ and
 $f_k^{p}(f_k^p(x_\ell^k))=f_k^p(x_\ell^k)$.

 For $j=1,\ldots,M$, we define $(a_j^k,b_j^k)$ as follows:
 if $\tilde a_j=y_\ell$ (for some $\ell$) is a surviving endpoint, we set
 $a_j^k$ to be the right-most point in
 $\mc O(x_\ell^k)\cap B_{\eps_0}(\tilde a_j)$
 and $b_j^k$ the left-most point in $\left(\mc O(x_\ell^k)\cup \{\tilde b_j\} \right)\cap B_{\eps_0}(\tilde b_j)$.
 If $\tilde a_j$ is not a surviving endpoint, then $\tilde b_j$ is necessarily surviving and we proceed similarly.
 Note that $(a_j^k,b_j^k)\in \Bcal_{f_k}$ and $d(a^k_j,\tilde a_j),d(b^k_j,\tilde b_j)<\eps_0\leq \eps/2$
 $(j=1\ldots,N$)
 and hence $\Bcal_f\ssq \bigcup_{j=1}^\ell B_{3\eps}((a_j^k,b_j^k))$.
 Since $\eps>0$ can be chosen arbitrarily small, this proves the desired statement.
\end{proof}

\subsection{Critical steps}\label{ss. critical steps}
Throughout this section, we consider continuous self-maps $f$ on $\I$
which are piecewise uniformly expanding (for the relation to piecewise monotone maps, see 
Remark~\ref{r.piecewise monotone from uniformly expanding}).
Recall that $f$ is referred to as \emph{piecewise uniformly expanding}, if there are finitely many
intervals $I_1,\ldots, I_n$ with $\I\ssq \bigcup_{\ell=1}^n I_\ell$
such that $f$ is uniformly expanding on each such interval, that is,
there is $\lambda>1$ such that $d(f(x),f(y))>\lambda \cdot d(x,y)$ whenever there is $\ell$ with
$x,y\in I_\ell$.
Given these intervals are maximal, the corresponding boundary points which do not
lie in $\partial \I$ coincide with the critical points of $f$.

Recall that a double point $(a,b)\in\Bcal_f$ is referred to as periodic, if both
$a$ and $b$ are periodic.
Our goal is to take a close look at periodic corner points (where $\mc O(a)=\mc O(b)$)
with $\Cri(f)\cap\Ocal(a)\neq\emptyset$.
We call this kind of periodic corner points (and hence steps, according to the previous section)
\emph{critical}.
Given a critical step $(a,b)$ of period $p$, we say that $f$ is
\emph{positive at $a$} whenever the image under $f^p$
of arbitrary small closed segments containing $a$
 is given by $[f^p(a),c]$ where $\I\ni c\neq f^p(a)=a$.
Clearly, if $f$ is not positive at $a$, then the image under
$f^p$ of an arbitrary small enough closed interval is given by $[c,f^p(a)]$ for some
$c\neq f^p(a)=a$ in $\I$.
In this situation, we say that $f$ is \emph{negative at $a$}.
For a periodic step $(a,b)\in \Bcal_f$ we say that $f$ is \emph{positive from $a$ to $b$}
if $f$ preserves orientation from $a$ to $b$ or if for some small enough
segment $J$ containing $a$ in its interior, we have $f^{n_{a,b}}(J)=[b,c]$ where
$\I\ni c\neq f^{n_{a,b}}(a)=b$.
In the complementary situation, we have that $f$ either reverses orientation from
$a$ to $b$ or we have that for an arbitrary small enough segment $J$ containing $a$ in its interior it holds
$f^{n_{a,b}}(J)=[c,b]$ for some $c\neq f^{n_{a,b}}(a)=b$ in $\I$.
In either case we say that $f$ \emph{is negative from $a$ to $b$}.
The following statement shows that in several situations $\Bcal_f$ detects
the periodicity of critical points explicitly.

\begin{lemma}\label{l.critpercor}
	Let $(a,b)\in\Bcal_f$ be a critical step of a transitive piecewise
	uniformly expanding map $f:\I\to\I$.
	If $f$ is negative at $a$ and positive at $b$, $(a,b)$ is accumulated
	from below.
\end{lemma}
\begin{proof}
We first show that for every $\eps>0$ we have that
there is $n\in\N$ and two distinct points $x,y\in [a,a+\eps]=I$
with $f^n(x),f^n(y)\in\mathcal{O}(a)$ (note that possibly $f^n(x)=f^n(y)=a$).
To that end, we may assume without loss of generality that
$\eps>0$ is small enough to guarantee that $(a+\eps,b')$ is a non-empty
subinterval of $(a,b)$, where $b'$ is the left-most point of $(a,b)\cap (f^{n_{a,b}}(I)\cup\{b\})$.
As $f$ is transitive, there clearly exists a transitive point $z\in I$.
In particular, there must be $n\geq1$ such that $f^n(z)\in (a+\eps,b')$.
Note that this necessarily gives
$\{a\}\ssq f^{n}(I)\cap\mathcal{O}(a)$ or
$\{b\}\ssq f^{n}(I)\cap\mathcal{O}(a)$.
In the first case, if $n$ is not a multiple of the minimal period $p$ of $a$, we are done since
$f^n(a)$ obviously lies in $\mc O(a)$ which would hence give two points in $\mc O(a)$.
If, however, $n$ is a multiple of $p$, we must have another point besides $a$ whose $n$-th image coincides with $a$
as $f$ is assumed to be negative at $a$.
The second case can be dealt with similarly.

Now, assume $n\in \N$ to be minimal with the discussed property and observe that the above argument also gives
\begin{align}\label{eq: no intersection with the hole}
 f^j(I)\cap (a,b)=\emptyset \qquad \text{ for } j = 1,\ldots, n_{a,b}-1,n_{a,b}+1,\ldots, n-1.
\end{align}
By definition of $n$, we hence have
$x_0\in I\setminus\{a\}$ with $f^{n}(x_0)\in\mathcal{O}(a)$ and such that
$f^j(x_0)\in\I\setminus (a,b)$ for every $j=1,\ldots, n_{a,b}-1,n_{a,b}+1,\ldots, n-1$,
due to \eqref{eq: no intersection with the hole}.
Clearly, given $\delta>0$ we can further guarantee that $y_0=f^{n_{a,b}}(x_0)$ is
$\delta$-close to $b$
by choosing the above $\eps$ small enough.
Then $(x_0,y_0)$ is a double point $2\delta$-close to $(a,b)$ and below $(a,b)$.
Since $\delta>0$ was arbitrary, this proves the statement.
\end{proof}

\begin{cor}\label{c.cri}
	Let $f:\I\to\I$ be continuous, transitive and piecewise uniformly expanding.
	Then $\Bcal_f$ has a step $(a,b)$ accumulated from below if and only if
	$(a,b)$ is a critical step and $f$ is negative at $a$ and positive at $b$.
\end{cor}
\begin{proof}
	The ``if''-part is given by the previous statement.
	For the other direction consider a periodic corner point $(a,b)\in\Bcal_f$
	accumulated from below.
	If it is hyperbolic, then it cannot be accumulated from below due to
	Lemma~\ref{l.periodiccorner}.
	Hence, it must be a critical periodic corner point.
	For a contradiction, suppose $f$ is negative at $a$ and negative at $b$ (the other cases can be dealt with similarly)
	and assume there is $(a',b')\in \Bcal_f$ with $a<a'<a+\delta$ and $b-\delta<b'<b$ for arbitrarily small $\delta>0$.
	We denote by $p$ the minimal period of $a$ and $b$.
	For small enough $\delta>0$, the negativity at $b$ implies that
	$f^p(b')\in (a',b')$ and that there is $\ell\in \N$ such that
	$f^{n_{a,b}+\ell\cdot p}(a')\in (a',b')$ since $f$ is piecewise uniformly expanding.
	For such $\delta$ we have $(a,a+\delta)\times (b-\delta,b)\ssq \Bcal_f^c$
	which finishes the proof.
\end{proof}

If $\I=\T$, we clearly have that if $b$ is the second coordinate of a step, then
it also is the first coordinate of the neighboring step of the associated stair.
In this way, we obtain the following statement
where the term \emph{negative slope region} of a piecewise uniformly expanding map refers to
a maximal interval in the complement of the critical points on which the map
reverses orientation.
The straightforward proof is left to the reader.

\begin{cor}\label{c.cricirc}
	Suppose $f:\T\to\T$ is a transitive piecewise uniformly expanding map.
	Then there is a step $(a,b)$ in $\Bcal_f$ which is accumulated from below if
	and only if $f$ has a critical periodic point which meets a negative
	slope region or it has a critical periodic point with an orbit supporting both a
	local maximum and a local minimum of $f$.
\end{cor}

It remains to study the case when $(a,b)\in \Bcal_f$ is a critical periodic corner
point not fulfilling the conditions of Lemma~\ref{l.critpercor}.
In this case, we obtain the following

\begin{lemma}\label{l.critpercor2}
	Let $f\: \I\to\I$ be continuous, transitive and piecewise uniformly expanding.
	Suppose $(a,b)\in\mathcal{B}_f$ is a critical periodic corner point such that
	$f$ is positive at $a$.
	Then there is a neighborhood $U\subseteq\Delta$ of $(a,b)$ and a sequence of maps
	$(f_n)_{n\in\N}$ converging uniformly to $f$ so that $\mathcal{B}_{f_n}\cap U=\emptyset$
	for every $n\in\N$.
	The same holds true if $f$ is negative at $b$.
\end{lemma}
\begin{proof}
Let $f$ be positive at $a$ (the proof of the other case works similarly)
and let $p$ denote the minimal period of $a$.
As $f$ is piecewise uniformly expanding and positive at $a$, there are $\eps_1,\eps_2,\delta>0$
such that for $I=(a-\eps_1,a+\eps_2)$,
$I^-=(a-\varepsilon_1,a)$ and $I^+=(a,a+\varepsilon_2)$

\begin{enumerate}
	\item $f^p(I^-)=f^p(I^+)=(a,a+\delta)$,\label{l.proof1}
	\item $f^p$ is uniformly expanding on $I^+$,\label{l.proof2}
	\item $f^j(I)\cap [a-\varepsilon_1-\delta,a+\varepsilon_2+\delta]=\emptyset$ \label{l.proof3}
		for $j=1,\ldots,p-1$.
\end{enumerate}

Observe the following: given $t\in(0,\varepsilon_2)$,
for every $x\in I$ there exit $m_x\in\N$ such that
$(f^{p}+t)^{m_x}(x)\in [a+\varepsilon_2,a+\varepsilon_2+\delta)=J$.  

To see this, note that for $x\in I$ with $y=(f^{p}+t)(x)\not\in J$ we have $y\in I^+$, due to \eqref{l.proof1}.
Since $f^p+t$ is uniformly expanding on $I^+$ (by \eqref{l.proof2}), the existence of the above $m_x$ follows.

Now, for big enough $n\in\N$, there is an orientation preserving homeomorphism $g_n$
with $g_n=\textrm{Id}$ on
$\I\setminus [a-\varepsilon_1-\delta,a+\varepsilon_2+\delta]$, $g_n(x)=x+1/n$ on
$(a-\eps_1,a+\delta)$ and $d_\infty(g_n,\textrm{Id})=1/n$.
Set $f_n=g_n\circ f$.
On the interval $I$, we have
$(g_n\circ f)^{p}=(f^p+1/n)$, due to \eqref{l.proof1} and \eqref{l.proof3}.
If $1/n<\eps_2$, the above observation implies that for every $x\in I$
we have $m_x\in\N$ such that $(g_n\circ f)^{p\cdot m_x}(x)=(f^{p}+1/n)^{m_x}(x)\in J$.

Consider now a neighborhood $U$ of $(a,b)\in\Delta$ such that for $(a',b')\in U$
we have $a'\in I,\ (a',b')\supset J, \mbox{ and }f^{n_{b,a}}(b')\in I$.
Then, given $(a',b')\in U$ we have for large enough $n$ that
$(g_n\circ f)^{p\cdot m_{a'}}(a')\in (a',b')$ and $(g_n\circ f)^{p\cdot m_{z}}(z)\in (a',b')$ where $z=f^{n_{b,a}}(b')$.
Hence, $U\subseteq \Delta\setminus \mathcal{B}_{g_n\circ f}$ for big enough $n$ which
proves the statement.
\end{proof}

For the next statement, we consider the space of continuous self-maps on $\I$ equipped
with the supremum metric $d_\infty$ and the space of non-empty closed subsets
of $\Delta$ endowed with the Hausdorff metric.

\begin{cor}\label{c.critpercor2}
	Let $f:\I\to\I$ be a continuous transitive piecewise uniformly expanding map and
	$(a,b)\in\mathcal{B}_f$ a critical periodic corner point such that
	$f$ is positive at $a$ or negative at $b$.
	Then the map $g\mapsto \Bcal_g$ is not continuous at $f$.
\end{cor}

Summing-up, we obtain the following statement concerning the sensitivity of the
bifurcation set to different dynamical behavior of the critical points.

\begin{thm}\label{t.bifsetandcrit}
	Assume that $f:\I\to\I$ is a continuous transitive piecewise uniformly expanding map.
	The following holds.
	\begin{enumerate}
		\item If $\Per(f)\cap\Cri(f)=\emptyset$, then every step is isolated from below.
			Further, in case $\Cri(f)$ is empty or only consists of transitive
			points, we get that $f$ is a continuity point of the map $g\mapsto\Bcal_g$.
		\item If $\Per(f)\cap\Cri(f)\neq\emptyset$, then there is at least one step
			accumulated from below or $f$ is a discontinuity point of $g\mapsto\Bcal_g$.
	\end{enumerate}
\end{thm}

\begin{rem}\label{r.piecewise monotone from uniformly expanding}
	According to a classical result of Parry \cite[Corollary 3]{Parry1966}, a transitive
	piecewise monotone map $f:[0,1] \to [0,1]$ is conjugate to a map of constant slope
	$\pm\beta$ where $\log\beta$ is the topological entropy of $f$.
	Further, it is well known that transitive continuous interval maps have positive
	entropy, see for instance \cite[Corollary 3.6]{BlockCoven1987}.
	Therefore, we can use Theorem \ref{t.bifsetandcrit} to infer Theorem~\ref{t.main B}
	for interval maps because topological properties of the bifurcation set are
	preserved under conjugation, see Section \ref{ss.invariants}.
	Concerning maps on the circle, \cite[Theorem C]{AlsedaMisiurewicz2015} provides
	the analogue statement of Parry's result.
	Moreover, for the fact that transitive non-minimal circle maps have positive
	entropy, see for example \cite[p. 267]{AlsedaLlibreMisiurewicz2000}.
\end{rem}


\section{Proof of Theorem \ref{t.main C}}\label{s.restricted tent map}

In order to emphasize the applicability of our results,
we now make use of the statements and techniques of the previous sections
to describe the dependence of the bifurcation set on the parameter of a particular family of interval maps.
The specific family we are interested in is given by the collection of \emph{restricted tent maps}
$(T_s)_{s\in(1,2]}$ which are defined via

\begin{minipage}[m]{0.68\textwidth}
	\begin{align*}
		T_s(x):=
		\begin{cases}
			1+s(x-c_s) & \textnormal { if } x\in[0,c_s]\\
			1-s(x-c_s) & \textnormal { if } x\in(c_s,1]
		\end{cases}
		\quad\;\textnormal{with } c_s:=1-\frac{1}{s}.\\\\
	\end{align*}
\end{minipage}
\begin{minipage}[m]{0.32\textwidth}
	\definecolor{zzzzzz}{rgb}{0.6,0.6,0.6}
	\begin{tikzpicture}[line cap=round,line join=round,>=triangle 45,x=3.4cm,y=3.4cm]
	\draw[color=black] (-0.05,0) -- (1.05,0);
	\foreach \x in {,1}
	\draw[shift={(\x,0)},color=black] (0pt,2pt) -- (0pt,-2pt) node[below] {\footnotesize $\x$};
	\draw[color=black] (0,-0.05) -- (0,1.05);
	\foreach \y in {,0.5,1}
	\draw[shift={(0,\y)},color=black] (2pt,0pt) -- (-2pt,0pt) node[left] {\footnotesize $\y$};
	\draw [dash pattern=on 1pt off 1pt,color=zzzzzz] (1-1/1.6,1)-- (1-1/1.6,0);
	\draw[line width=0.5pt,color=zzzzzz, smooth,samples=100,domain=0:1] plot(\x,{(\x)});
	\draw[line width=1pt,smooth,samples=100,domain=0:1] plot(\x,{1-1.6*abs((\x)-1+1/1.6)});
	\draw (0.32,-0.02) node[anchor=north west] {$c_s$};
	\draw (0.45,0.98) node[anchor=north west] {$T_s$};
	\end{tikzpicture}
\end{minipage}
The above figure depicts $T_s$ for $s=1.6$.
It is not difficult to show that each $T_s$ is conjugate to the tent map $t_s:[0,1]\to [0,1]$
given by $x\mapsto s(1/2-\abs{x-1/2})$ restricted to the
interval $[t_s^2(1/2),t_s(1/2)]$ for $s\in(1,2]$.
Moreover, it is well known that $T_s$ is transitive if and only if $s\in[\sqrt{2},2]$,
see e.g.\ \cite[Lemma 8.1]{Shultz2007}.
Using the relation between $T_s$ and $t_s$, the following holds.

\begin{thm}[{\cite[Theorem 7]{BrucksMisiurewicz1996} and
	\cite[Lemma 5.5]{CovenKanYorke1988}}]\label{thm: transitivity and periodicity of c_s}\label{t.family}
	For almost every $s\in[\sqrt{2},2]$ we have that the critical point $c_s$
	is transitive. Further, $c_s$ is a periodic point for a dense set of parameters.
\end{thm}

By means of this result, we will obtain

\begin{thm}\label{t.bifsetfamiliy}

Consider the family of restricted tent maps $(T_s)_{s\in[\sqrt{2},2]}$.
Then

\begin{enumerate}

\item If $c_{s_0}$ is transitive, the steps of $\mathcal{B}_{T_{s_0}}$ are isolated from below and
the map $s\mapsto \mathcal{B}_{T_s}$ is continuous in $s_0$.\label{thmC: 1}

\item If $c_{s_0}$ is periodic, the map $s\mapsto \mathcal{B}_{T_s}$ is not continuous in $s_0$.\label{thmC: 2}

\item There exists a dense set of parameters $\mathcal{J}\ssq [\sqrt{2},2]$ so that $c_s$ is periodic and
there are steps of $\mathcal{B}_{T_s}$ which are accumulated from below whenever $s\in\mathcal{J}$.\label{thmC: 3}

\end{enumerate}

\end{thm}

We devote the rest of the section to the proof of this statement.
Clearly, point \eqref{thmC: 1} follows from Theorem~\ref{t.bifsetandcrit}.
In view of Lemma~\ref{l.critpercor}, point \eqref{thmC: 3} can be deduced from the next proposition.

\begin{prop}
	There is a dense set of parameters $\mathcal{J}\subset [\sqrt{2},2]$
	such that for every $s\in\mathcal{J}$ we have a step
	$(a,b)\in\mathcal{B}_{T_s}$ where $T_s$ is negative at $a$ and positive at $b$.
\end{prop}

\begin{proof}

Observe that for all $s\in(1,2]$ the point $x_s=1-1/s-1/s^2+1/s^3$ verifies
$x_s\in(0,c_s)$,
$T_s(x_s)\in (c_s,1)$ and $T_s^2(x_s)=c_s$.
By choosing $y_s$ sufficiently close and to the right of $x_s$, we can guarantee
that $0<y_s <T_s^2(y_s)<c_s<T_s(y_s)<T_s^3(y_s)<1$.
In particular, $T_s$ is order preserving in $y_s$ as well as in $T_s^2(y_s)$ and order reversing in $T_s(y_s)$ as
well as in $T_s^3(y_s)$.

Given $s\in [\sqrt{2},2]$, by Theorem~\ref{t.family} there is an
arbitrarily close $s'$ such that $c_{s'}$ is transitive.
Observe that there is $n\in\N$ with
$0<T_{s'}^{n}(c_{s'}) <T_{s'}^{n+2}(c_{s'})<c_{s'}<T_{s'}^{n+1}(c_{s'})<T_{s'}^{n+3}(c_{s'})<1$
(pick $n$ such that $T_{s'}^n(c_{s'})$ is sufficiently close to $y_{s'}$).
Now, Theorem~\ref{t.family} allows to pick $s''$ such that $c_{s''}$ is periodic
and such that $s''$ is sufficiently close to $s'$ to guarantee
$0<T_{s''}^{n}(c_{s''}) <T_{s''}^{n+2}(c_{s''})<c_{s''}<T_{s''}^{n+1}(c_{s''})<T_{s''}^{n+3}(c_{s''})<1$.
Hence, at $s''$ we have $a'<b'<c'<d'$ where $a'=T^n_{s''}(c_{s''})$, $b'=T^{n+2}_{s''}(c_{s''})$,
$c'=T^{n+1}_{s''}(c_{s''})$ and $d'=T^{n+3}_{s''}(c_{s''})$.
Note that either $T_{s''}$ is negative at $a'$, positive at $b'$, negative at $c'$ and positive at $d'$
or the other way around, that is,
$T_{s''}$ is positive at $a'$, negative at $b'$, positive at $c'$ and negative at $d'$.

In the first case, choose $a\in \mc O(c_{s''})\cap [a',b']$ to be such that
$T_{s''}$ is negative at $a$ and $T_{s''}$ is positive at each element of $(a,b']\cap \mc O(c_{s''})$.
Choose $b$ to be the smallest element of $(a,b']\cap \mc O(c_{s''})$.
Then $(a,b)$ is a periodic corner point with $T_{s''}$ negative at $a$ and positive at $b$.
In the second case (when $T_{s''}$ is positive at $a$ etc.), we obtain a similar statement
by dealing with $b'$ instead of $a'$ and $c'$ instead of $b'$.

As $s\in [\sqrt{2},2]$ is arbitrary and $s''$ can
be chosen arbitrarily close to $s$, the statement follows.
\end{proof}

With regards to point \eqref{thmC: 2} of Theorem~\ref{t.bifsetfamiliy}, observe that
if $c_s$ is periodic, then $0=T_s^2(c_s)$ is periodic, too.
Therefore, as $T_s$ is clearly positive at $0$, Lemma~\ref{l.critpercor2} yields that $T_s$ is a discontinuity point
of the mapping $f\mapsto \Bcal_f$.\footnote{Note that formally speaking, as our present definition of $\Bcal_f$ excludes
points with coordinate entries equal to zero, we could not apply Lemma~\ref{l.critpercor2}.
However, this issue is of a rather formal nature (see also the remark in the
introduction) and will further not play a role in the discussion of the discontinuity
of $s\mapsto \Bcal_{T_s}$ as this discussion has to be carried out explicitly anyway.}
In Proposition~\ref{prop: discontinuity of Ts at periodic parameters} (see below),
we will show that this discontinuity is already visible within the family
$(T_s)_{s\in[\sqrt{2},2]}$.
To see this, we first make some simple technical observations.

Note that for all $x\in [0,1]$ and all $\ell\in \N$ with $T^j_s(x)\neq c_s$ ($j=0,\ldots,\ell$),
we have that $T_s^{\ell+1}(x)$ is differentiable with respect to $s$ (as well as $x$) and
\begin{align*}
 \frac{d}{ds} T^{\ell+1}_s(x)=\left(\frac{\partial}{\partial s} T_s\right)\left(T_s^\ell(x)\right)+
 T_s'\left(T_s^\ell(x)\right)\cdot \frac{d}{ds} T^\ell_s(x).
\end{align*}
As
\begin{align}\label{eq: derivatives of Ts}
 T_s'(x)=
	\begin{cases}
             s & \text{if } x\in [0,c_s)\\
             -s & \text{if } x\in (c_s,1]
        \end{cases}
 \qquad \text{and} \qquad
 \frac{\partial}{\partial s} T_s(x)=
	\begin{cases}
		-1+x& \text{if } x\in [0,c_s)\\
                1-x& \text{if } x\in (c_s,1]
        \end{cases},
\end{align}
we hence have
\begin{align}\label{eq: recursive formula for d/ds T}
 \frac{d}{ds} T^{\ell+1}_s(x)=
 \begin{cases}
	-1+T_s^\ell(x)+s\cdot \frac{d}{ds} T^\ell_s(x)
	&\text{if } T^\ell_s(x)\in [0,c_s)\\
        1-T_s^\ell(x)-s\cdot \frac{d}{ds} T^\ell_s(x)
        &\text{if } T^\ell_s(x)\in (c_s,1]
 \end{cases}.
\end{align}

\begin{prop}\label{prop: derivative of the critical point}
 Consider the family of restricted tent maps $(T_s)_{s\in[\sqrt{2},2]}$.
 If $c_{s_0}$ is periodic, then $\frac{d}{ds} (T_s^{n_{0,c_{s_0}}}(0)-c_s)\neq 0$ at $s=s_0$.

 Moreover, if $\frac{d}{ds} (T_s^{n_{0,c_{s_0}}}(0)-c_s)> 0$, then
 $\left(T_s^{n_{0,c_{s_0}}}\right)'(0)=-s^{n_{0,c_{s_0}}}$.
 Otherwise $\left(T_s^{n_{0,c_{s_0}}}\right)'(0)=s^{n_{0,c_{s_0}}}$.
\end{prop}
\begin{proof}
First, note that $T_{s_0}^{j}(0)\neq c_{s_0}$ for all $j=0,\ldots,n_{0,c_{s_0}}-1$
so that the above expressions are indeed differentiable.
For the first part, we have to consider three cases.

\noindent
\emph{Case 1: $s_0=(1+\sqrt{5})/2$.}
This is the only case in which $n_{0,c_{s_0}}=1$.
By \eqref{eq: derivatives of Ts}, we have $\frac{d}{ds} (T_s(0)-c_s)=-1+1/s^2< 0$.

In the remaining cases, we will show that, in fact,
\begin{align}\label{eq: d/ds T bigger than d/ds c}
	\left|\frac{d}{ds} T_s^{n_{0,c_{s_0}}}(0)\right|\geq 1/(s-1)>\left |\frac{d}{ds} c_s \right|=1/s^2
\end{align}
at $s=s_0$.
To that end, note that \eqref{eq: recursive formula for d/ds T} yields that if
\begin{align}\label{eq: d/ds T bigger than d/ds c second}
	\left|\frac{d}{ds} T_s^{j}(0)\right|\geq 1/(s-1)
\end{align}
for some $j=j_0<n_{0,c_{s_0}}$,
then \eqref{eq: d/ds T bigger than d/ds c second} also holds for $j=j_0+1$.
Hence,
it suffices to show that there is
some $j\leq n_{0,c_{s_0}}$ for which \eqref{eq: d/ds T bigger than d/ds c second} holds
in order to prove \eqref{eq: d/ds T bigger than d/ds c} and hence the first part of the statement.

\noindent
\emph{Case 2: $s_0\in [3/2,2]\setminus \{(1+\sqrt{5})/2\}$.}
By an immediate computation, we have $\left|\frac{d}{ds} T_s^{2}(0)\right|=2s-1$ at all $s\in[\sqrt{2},2]$ and
hence $\left|\frac{d}{ds} T_s^{2}(0)\right|\geq 1/(s-1)$ for all $s\geq 3/2$.

\noindent
\emph{Case 3: $s_0< 3/2$.}
Observe that in this case we have $T_{s_0}(0),T_{s_0}^2(0),T_{s_0}^3(0)\in(c_{s_0},1]$.
It hence suffices to show
\eqref{eq: d/ds T bigger than d/ds c second} for $j=4$.
Now, if $T_{s}(0),T_{s}^2(0),T_{s}^3(0)\in(c_{s},1]$, we have $T_{s}^4(0)={s}^4-{s}^3-s^2+s$.
Therefore, $\frac{d}{ds} T_{s}^4(0)=4s^3-3s^2-2s+1$.
By elementary means, we see that this indeed gives $\frac{d}{ds} T_{s}^4(0)\geq 1/(s-1)$ for all $s<3/2$
which finishes the proof of the first part.

For the second part,
notice that if $\left|\frac{d}{ds} T_s^{j}(0)\right|\geq 1/(s-1)$ for some $j< n_{0,c_{s}}$ and
$\frac{d}{ds} T_s^{j}(0)$ is positive (negative), then
\eqref{eq: recursive formula for d/ds T} gives that
$\frac{d}{ds} T_s^{j+1}(0)$ is positive (negative) if and only if
$T_s^{j}(0)\in [0,c_{s})$.
With this in mind, an inspection of the above cases shows that whenever
we have $\frac{d}{ds} (T_s^{n_{0,c_{s_0}}}(0)-c_s)< 0$ it holds that
$\#\{ j\in\{0,\ldots,n_{0,c_{s_0}}-1\} \: T_{s_0}^j(0)\in (c_{s_0},1]\}$
is even.
Hence, in this case we obtain by the chain rule that
$\left(T_s^{n_{0,c_{s_0}}}\right)'(0)=T_s'\left(T_s^{n_{0,c_{s_0}}-1}(0)\right)\cdot \ldots\cdot T_s'(0)=s^{n_{0,c_{s_0}}}$.
The other case is similar.
\end{proof}

\begin{prop}\label{prop: discontinuity of Ts at periodic parameters}
Consider the family of restricted tent maps $(T_s)_{s\in[\sqrt{2},2]}$.
If $c_{s_0}$ is periodic, then the map $s\mapsto \mathcal{B}_{T_s}$ is not continuous in $s_0$.
\end{prop}
\begin{proof}
 Let $p$ be the minimal period of $0$ under $T_{s_0}$ and
 let $b$ be the element in $\mc O(0)\setminus \{0\}$ which is the closest to $0$.
 Note that the horizontal segment $H=\{(a,b)\:0<a<b\}$ is entirely contained in $\Bcal_{T_{s_0}}$.
 Our goal is to show that there is some $\eps_0>0$ such that $(0,\eps_0)\times B_{\eps_0}(b)\ssq \Bcal_{T_s}^c$
 for $s$ sufficiently close to $s_0$.
 This clearly proves the statement.

 By Proposition~\ref{prop: derivative of the critical point}, we either have
 $\frac{d}{ds} (T_{s}^{n_{0,c_{s_0}}}(0)-c_s)>0$ and $\left(T_s^{n_{0,c_{s_0}}}\right)'(0)=-s^{n_{0,c_{s_0}}}$ or
 $\frac{d}{ds} (T_{s}^{n_{0,c_{s_0}}}(0)-c_s)<0$ and $\left(T_s^{n_{0,c_{s_0}}}\right)'(0)=s^{n_{0,c_{s_0}}}$.
 Set $\mc I'=(s_0-\delta,s_0)$ for some $\delta>0$.\footnote{Note that $s_0$ is
 necessarily different from $\sqrt{2}$, since $T_{\sqrt{2}}(0)$
 coincides with the unique fixed point of $T_{\sqrt{2}}$ which gives that $c_{\sqrt{2}}$ is not periodic.
 Hence, $\mc I'$ is always an interval which has a non-trivial intersection with $[\sqrt{2},2]$.}
 Observe that $T_s^j(x)\neq c_s$ $(x\in[0,\eps],\ s\in \mc I',\ j=0,\ldots,p)$
 and hence, in fact, $T_s^p(x)-T_s^p(y)=s^p(x-y)$ for all $x,y\in[0,\eps]$ and each $s\in \mc I'$
 whenever $\eps>0$ and $\delta$ are sufficiently small.

 W.l.o.g.\ we may assume $\eps< b/(4s^p)$ as well as
 $T_s^p(0),T_s^{n_{b,0}}(b) <\eps/2$ ($s\in \mc I'$),
 where $n_{b,0}<p$ is such that $T_{s_0}^{n_{b,0}}(b)=0$.
 Clearly, $T_s^{p}(x)=T_s^p(0)+s^{p} x<\eps+s^p\eps/2<b/2$ whenever $x\in [0,\eps]$.

 Altogether, the above shows that for each $x\in(0,\eps]$ and every $s\in \mc I'$ we have some $\ell\in \N$ with
 $T_s^{\ell p}(x)\in (\eps,b/2)$.
 Accordingly, if $|x|<\eps$ and $|b-y|<\eps/(2s_0^{n_{b,0}})$, we have some $\ell_x,\ell_y\in \N$ such that
 $T_s^{\ell_x p}(x),T_s^{\ell_y p}(y)\in (\eps,b/2)$ and hence $(x,y)\in \Bcal_{T_s}^c$ for every $s\in \mc I'$.
 Therefore, $\Bcal_{T_s}\cap \left((0,\eps)\times B_{\eps/(2s_0^{n_{b,0}})}(b) \right)=\emptyset$ for each $s\in \mc I'$.
\end{proof}
\bibliography{open-systems-literature}{}

\begin{thebibliography}{10}

\bibitem{DemersYoung2005}
M.F. Demers and L.-S. Young.
\newblock Escape rates and conditionally invariant measures.
\newblock {\em Nonlinearity}, 19(2):377--397, 2005.

\bibitem{KellerLiverani2009}
G.~Keller and C.~Liverani.
\newblock {Rare Events, Escape Rates and Quasistationarity: Some Exact
  Formulae}.
\newblock {\em Journal of Statistical Physics}, 135(3):519--534, 2009.

\bibitem{Bunimovich2011}
L.A. Bunimovich and A.~Yurchenko.
\newblock Where to place a hole to achieve a maximal escape rate.
\newblock {\em Israel Journal of Mathematics}, 182(1):229--252, 2011.

\bibitem{Dettmann2011}
C.P. Dettmann.
\newblock {\em Recent Advances in Open Billiards with Some Open Problems},
  volume~16 of {\em World Scientific series on nonlinear science}, pages
  195--218.
\newblock World Scientific, 2011.

\bibitem{AltmannPortelaTel2013}
E.G. Altmann, J.S.E. Portela, and T.~T\'el.
\newblock Leaking chaotic systems.
\newblock {\em Reviews of Modern Physics}, 85:869--918, 2013.

\bibitem{DemersFernandez2016}
M.~Demers and B.~Fernandez.
\newblock {Escape Rates and Singular Limiting Distributions for Intermittent
  Maps with Holes}.
\newblock {\em Transactions of the American Mathematical Society},
  368(7):4907--4932, 2016.

\bibitem{PollicottUrbanski2017}
M.~Pollicott and M.~Urba{\'n}ski.
\newblock {\em Open conformal systems and perturbations of transfer operators}.
\newblock Springer, 2017.

\bibitem{BruinDemersTodd2018}
H.~Bruin, M.F. Demers, and M.~Todd.
\newblock Hitting and escaping statistics: mixing, targets and holes.
\newblock {\em Advances in Mathematics}, 328:1263--1298, 2018.

\bibitem{CarminatiTiozzo2011}
C.~Carminati and G.~Tiozzo.
\newblock The bifurcation locus for numbers of bounded type.
\newblock {\em Preprint arXiv:1109.0516}, 2011.

\bibitem{Sidorov2014}
N.~Sidorov.
\newblock {Supercritical Holes for the Doubling Map}.
\newblock {\em Acta Mathematica Hungarica}, 143(2):298--312, 2014.

\bibitem{HareSidorov2014}
K.G. Hare and N.~Sidorov.
\newblock On cycles for the doubling map which are disjoint from an interval.
\newblock {\em Monatshefte f{\"u}r Mathematik}, 175(3):347--365, 2014.

\bibitem{GlendinningSidorov2015}
P.~Glendinning and N.~Sidorov.
\newblock The doubling map with asymmetrical holes.
\newblock {\em Ergodic Theory and Dynamical Systems}, 35(4):1208--1228, 2015.

\bibitem{Groeger2015}
M.~Gr{\"o}ger.
\newblock {\em Examples of dynamical systems in the interface between order and
  chaos}.
\newblock PhD thesis, Universit{\"a}t Bremen, 2015.

\bibitem{Clark2016}
L.~Clark.
\newblock The {$\beta$}-transformation with a hole.
\newblock {\em Discrete \& Continuous Dynamical Systems - A}, 36(3):1249--1269,
  2016.

\bibitem{KalleKongLangeveldLi2018}
C.~Kalle, D.~Kong, N.~Langeveld, and W.~Li.
\newblock The {$\beta$}-transformation with a hole at 0.
\newblock {\em Ergodic Theory and Dynamical Systems}, 2018.
\newblock To appear.

\bibitem{Urbanski1986}
M.~Urba{\'n}ski.
\newblock {On Hausdorff dimension of invariant sets for expanding maps of a
  circle}.
\newblock {\em Ergodic Theory and Dynamical Systems}, 6:295--309, 1986.

\bibitem{Martens1994}
M.~Martens.
\newblock {Distortion results and invariant Cantor sets of unimodal maps}.
\newblock {\em Ergodic Theory and Dynamical Systems}, 14(2):331--349, 1994.

\bibitem{MisiurewiczSzlenk1980}
M.~Misiurewicz and W.~Szlenk.
\newblock Entropy of piecewise monotone mappings.
\newblock {\em Studia Mathematica}, 67(1):45--63, 1980.

\bibitem{AlsedaLlibreMisiurewicz2000}
L.~Alsed{\`{a}}, J.~Llibre, and M.~Misiurewicz.
\newblock {\em Combinatorial Dynamics and Entropy in Dimension One}.
\newblock World Scientific, 2000.

\bibitem{Sharkovsky1964}
A.N. Sharkovsky.
\newblock {Fixed points and the center of a continuous mapping of the line into
  itself (Ukrainian)}.
\newblock {\em Dopovidi Akademi{\"{i}} nauk Ukra{\"{i}}ns'ko{\"{i}} RSR},
  7:865--868, 1964.

\bibitem{BlockCoppel1992}
L.S. Block and W.A. Coppel.
\newblock {\em Dynamics in One Dimension}.
\newblock Lecture Notes in Mathematics. Springer Berlin, 1992.

\bibitem{CovenMulvey1986}
E.M. Coven and I.~Mulvey.
\newblock Transitivity and the centre for maps of the circle.
\newblock {\em Ergodic Theory and Dynamical Systems}, 6(1):1--8, 1986.

\bibitem{AuslanderKatznelson1979}
J.~Auslander and Y.~Katznelson.
\newblock Continuous maps of the circle without periodic points.
\newblock {\em Israel Journal of Mathematics}, 32(4):375--381, 1979.

\bibitem{MeloStrien1993}
W.~de~Melo and S.~van Strien.
\newblock {\em One-dimensional dynamics}, volume~25 of {\em A Series of Modern
  Surveys in Mathematics}.
\newblock Springer, 1993.

\bibitem{Parry1966}
W.~Parry.
\newblock Symbolic dynamics and transformations of the unit interval.
\newblock {\em Transactions of the American Mathematical Society},
  122(2):368--378, 1966.

\bibitem{BlockCoven1987}
L.~Block and E.M. Coven.
\newblock {Topological Conjugacy and Transitivity for a Class of Piecewise
  Monotone Maps of the Interval}.
\newblock {\em Transactions of the American Mathematical Society},
  300(1):297--306, 1987.

\bibitem{AlsedaMisiurewicz2015}
L.~Alsed{\`{a}} and M.~Misiurewicz.
\newblock Semiconjugacy to a map of a constant slope.
\newblock {\em Discrete \& Continuous Dynamical Systems - B},
  20(10):3403--3413, 2015.

\bibitem{Shultz2007}
F.~Shultz.
\newblock {Dimension groups for interval maps II: the transitive case}.
\newblock {\em Ergodic Theory and Dynamical Systems}, 27(4):1287--1321, 2007.

\bibitem{BrucksMisiurewicz1996}
K.~Brucks and M.~Misiurewicz.
\newblock The trajectory of the turning point is dense for almost all tent
  maps.
\newblock {\em Ergodic Theory and Dynamical Systems}, 16(6):1173--1183, 1996.

\bibitem{CovenKanYorke1988}
E.M. Coven, I.~Kan, and J.A. Yorke.
\newblock {Pseudo-Orbit Shadowing in the Family of Tent Maps}.
\newblock {\em Transactions of the American Mathematical Society},
  308(1):227--241, 1988.

\end{thebibliography}
\bibliographystyle{unsrt}
\vspace{-0.38cm}
\end{document}